\documentclass[10pt,reqno,a4paper]{amsart}

\usepackage[english]{babel}
\usepackage[OT2,OT1]{fontenc}
\usepackage{amsmath,amssymb}

\usepackage[colorlinks,citecolor=red]{hyperref} 
\usepackage[numbers,sort]{natbib} 

\usepackage{anysize}

\numberwithin{equation}{section}

\newtheorem{theorem}{Theorem}[section]
\newtheorem{lemma}{Lemma}[section]
\newtheorem{proposition}{Proposition}[section]
\newtheorem{corollary}[theorem]{Corollary}
\theoremstyle{definition}
\newtheorem{definition}{Definition}[section]
\newtheorem{example}[theorem]{Example}
\theoremstyle{remark}
\newtheorem{remark}{Remark}[section]

\usepackage{enumerate}

\title[Fundamental solutions and decay rates] 
      {Fundamental solutions and decay of fully non-local problems}

\author[Juan C. Pozo and Vicente Vergara]{}


 \email{juan.pozo@ufrontera.cl}
 \email{vicente.vergara@udec.cl}

\thanks{$^*$The first author is partially supported by Fondecyt grant 11160295}
\thanks{$^\dagger$The second author is partially supported by Fondecyt grant 1150230}

\thanks{$^\ddagger$ Corresponding author: Juan C. Pozo}

\begin{document}
\maketitle

\centerline{\scshape Juan C. Pozo$^{*,\ddagger}$}
\medskip
{\footnotesize
 \centerline{Departamento de Matem\'aticas y Estad\'isticas, Facultad de Ingenier\'ia y Ciencias.}
   \centerline{Universidad de La Frontera, Temuco, Chile.}
} 

\medskip

\centerline{\scshape Vicente Vergara$^{\dagger}$}
\medskip
{\footnotesize
 \centerline{Departamento de Matem\'aticas, Facultad de Ciencias F\'isicas y Matem\'aticas}
\centerline{Universidad de Concepci\'on,  Concepci\'on, Chile.}
}

\bigskip

 \centerline{(Communicated by the associate editor name)}

\begin{abstract}
In this paper, we study a fully non-local reaction-diffusion equation which is non-local both in time and space. We apply subordination principles to construct the fundamental solutions of this problem, which we use to find a representation of the mild solutions. Moreover, using techniques of Harmonic Analysis and Fourier Multipliers, we obtain the temporal decay rates for the mild solutions.
\end{abstract}

\section{Introduction}

This paper is devoted to study temporal decay estimates of solutions of a non-local in time and space reaction-diffusion problem. More precisely, we consider the equation
\begin{equation}\label{Eq:Intro1}
\partial_t(k*[u(t,x)-u_0(x)])+(-\Delta)^{\rho/2} u(t,x)=f(t,x),\quad t > 0,\quad x\in \mathbb{R}^d,
\end{equation}
\noindent with initial condition,
\begin{equation}\label{Eq:Intro2}
u(0,x)=u_0(x),\quad x\in\mathbb{R}^d,
\end{equation}
\noindent where $u_0$ and $f$ are given functions and  $(k*v)$ denotes the convolution product  on the positive halfline $\mathbb{R}_+:=[0,\infty)$ with respect to time variable, this is $(k*v)(t)=\int_0^t k(t-s)v(s)ds,$ with $t\ge 0$. The operator $(-\Delta)^{\rho/2}$ with $\rho>0$ is  known in the literature as fractional Laplacian and $k$ is a kernel of type $(\mathcal{PC})$, by which we mean that the following condition is satisfied.

\medbreak

\begin{enumerate}
\item [$(\mathcal{PC})$] $k \in L_{1,loc}(\mathbb{R}_+)$ is nonnegative and nonincreasing, and there exists a kernel $\ell\in L_{1,loc}(\mathbb{R}_+)$ such that $(k*\ell) = 1$ in $(0, \infty)$. In this case we also write $(k,\ell)\in(\mathcal{PC})$.
\end{enumerate}

\medbreak
 
We point out that the kernels of type $(\mathcal{PC})$ are divisors of the unit with respect to the temporal convolution. These kernels are also called {\it Sonine kernels} and they have been successfully used to study integral equations of first kind in the spaces of H\"older continuous, Lebesgue and Sobolev functions, see \cite{Car-Fio-Ten-2017}. 

Further, the condition $(\mathcal{PC})$ covers several interesting integro-differential operators with respect to time that appear in the context of subdiffusion processes.  

For instance, a very important example of $(k,\ell)\in(\mathcal{PC})$ is given by the pair $(g_{1-\alpha},g_\alpha)$ with $\alpha\in(0,1)$, where $g_\beta$ is the standard notation for the function
\[
g_\beta(t)=\dfrac{t^{\beta-1}}{\Gamma(\beta)},\quad t>0,\quad \beta>0.
\]

In this case the term $\partial_t(k*v)$ becomes the Riemman-Liouville fractional derivative $\partial_t^\alpha v$ of order $\alpha\in(0,1)$. The Riemman-Liouville fractional derivative is closely related with a class of Montroll-Weiss continuous time random walk models and it has become one of the standard physics approaches to model anomalous diffusion processes. The details of the derivation of these equations from physics principles and for further applications of such models can be found in \cite{Met-Kla-2000}. Furthermore, if $\rho=2$ in \eqref{Eq:Intro1}-\eqref{Eq:Intro2} the corresponding {\it mean square displacement}, (which is an important quantity that measures the dispersion of random processes and that describes how fast particles diffuse), behaves like $t^\alpha$ for large times, see \cite[Lemma 2.1]{Kem-Sil-Ver-Zach-2016}. For this reason, sometimes, this problem is referred as {\it subdiffusion equation}. 

With $\rho\in(0,2]$ and this kernel, \eqref{Eq:Intro1}-\eqref{Eq:Intro2} is called  {\it fully nonlocal diffusion equation} and it has been recently studied in \cite{Kem-Sil-Zach-2017,Kim-Lim-2016}. In both papers, the authors need several technical results about the so-called  Mittag-Leffler  and Fox $H$-functions to obtain the asymptotic behavior of the mild solutions. However, this approach seems not to be very helpful (or easy) to derive the large-time behavior of solutions to equations with other nonlocal in time operators. In this situation, to obtain the analogues functions to the Mittag-Leffler and Fox $H$-functions could be a very hard task.


Another interesting and important example of kernels $(k,\ell)\in(\mathcal{PC})$ is given by 
\begin{equation}\label{k:dist:order}
k(t)=\int_0^1 g_\alpha(t)\nu(\alpha)d\alpha , \quad t>0,
\end{equation}
\noindent where $\nu$ is a continuous non-negative function that not vanishes in a set of posi\-tive measure. Under appropriate conditions on $\nu$, the existence of a function $\ell \in L^1_{loc}(\mathbb{R}^+)$ such that $k*\ell=1$, has been established in \cite[Proposition 3.1]{Koch-2008}. In this case the operator $\partial_t (k*\cdot)$ is a so-called {\it operator of distributed order}, and \eqref{Eq:Intro1}-\eqref{Eq:Intro2} is an example of a so-called {\it ultraslow diffusion equation} if $\nu(0)=0$ and $\rho=2$ (see \cite{Koch-2008}). The special cases $\nu(\alpha)=\alpha^n$  with $n\in\mathbb{N}$, are discussed in Example \ref{Ex:4} below. Ultraslow diffusion equations have been successfully used  in physical literature for modeling diffusion with a logarithmic growth of the mean square displacement, see \cite[Theorem 4.3]{Koch-2008}. 

The fractional powers of the Laplacian operator arise naturally in different contexts. To mention a few of them, combustion theory \cite{Caff-Roq-Sir-2010}, dislocation processes of mechanical systems \cite{Imb-Mon-2008,Imb-Mon-Rou-2008}, among others. Further, it is a well-known fact from Probability theory that fractional Laplacian is the standard example of a non-local operator that generates a markovian $C_0$-semigroup, see \cite{Sat-1999}.

Another context where equations of the form \eqref{Eq:Intro1}-\eqref{Eq:Intro2} and nonlinear variants of them appear is the modeling of dynamic processes in materials with memory. Examples are given by the theory of heat conduction with memory, see, e.g., \cite{Pru-1993} and the references therein as well as \cite{Ver-Zach-2015}, and the diffusion of fluids in porous media with memory, cf. \cite{Ign-Ros-2009}. Given the condition $(\mathcal{PC})$, the problem \eqref{Eq:Intro1}-\eqref{Eq:Intro2} can be reformulated as an abstract Volterra equation on the positive half-line with a completely positive kernel; this can be seen by convolving the partial differential equation with the kernel $\ell$. There exists a substantial amount of work on such abstract Volterra and integro-differential equations since the 1970s, in particular on existence and uniqueness, regularity, and long-time behavior of solutions, see, for instance, \citep{Cle-Noh-1981, Gri-Lon-Sta-1990,Zach-2005,Zach-2008}, and the monograph \cite{Pru-1993}.

\medbreak

One of the main objectives of this paper is to prove sharp estimates for the temporal decay of solutions of the problem \eqref{Eq:Intro1}-\eqref{Eq:Intro2}. We point out that for non-local in space diffusion equations, in particular space-fractional diffusion equations, corresponding results have been obtained recently, see e.g. \cite{Vaz-2018,Bar-Per-Sor-Val-2014}. Concerning to non-local in time diffusion problems, the homogeneous case and kernels of type $(\mathcal{PC})$ have been studied recently in \cite{Ver-Zach-2015,Kem-Sil-Ver-Zach-2016}, where the authors  have obtained optimal rates of decay. 

\medbreak

We briefly explain how we obtain our main results. We first have by means of the subordination principle in the sense of Bochner and in the sense of Pr\"uss, see \cite{Boch-1949,Boch-1955, Pru-1993} a fundamental solution $Z(t,x)$ of \eqref{Eq:Intro1}-\eqref{Eq:Intro2}  with $f=0$ and $u_0=\delta_0$. Next, by the variations of parameters formula for Volterra equations, we define the mild solution $u$ of \eqref{Eq:Intro1}-\eqref{Eq:Intro2} as
\begin{equation}\label{Mild:Sol:0}
u(t,\cdot)= Z(t,\cdot)\star u_0(\cdot) + \int_0^t Y(t-s,\cdot)\star f(s,\cdot)ds,
\end{equation}
where the symbol $\star$ stands for the convolution of two functions in $L_p(\mathbb{R}^d)$ and the kernel $Y$ solves the Volterra equation of the first kind $\bigl(Y(\cdot,x)\ast k\bigr)(t) = Z(t,x)$. To derive decay estimates for $Z(t,\cdot)\star u_0(\cdot)$ and $\int_0^t Y(t-s,\cdot)\star f(s,\cdot)ds$, we adapt a method proposed recently in \cite[Section 5]{Kem-Sil-Ver-Zach-2016}, for the case $\rho =2$ and $f=0$, which is based on tools from the harmonic analysis and a careful estimation of the Fourier symbol $\tilde{Z}(t,\xi)$ of $Z$ with respect to the spatial variable. In this paper, we develop this method for \eqref{Eq:Intro1}-\eqref{Eq:Intro2}. At this point, two relaxation functions $s_{\mu}$ and $r_{\mu}$ arise from the well-known theory of Volterra equations see e.g. \cite[Chapter 2]{Gri-Lon-Sta-1990} and \cite{Cle-Noh-1979},  and they are the key to obtain the long-time behavior of solution $u$ of \eqref{Eq:Intro1}-\eqref{Eq:Intro2} given by \eqref{Mild:Sol:0}. 


\medskip

The paper is organized as follows. In Section \ref{relaxation} we collect some properties of the relaxation functions $s_{\mu}$ and $r_{\mu}$, which play an important role to derive the long-time behavior of solution to \eqref{Eq:Intro1}. Section \ref{Section:Main} is devoted to obtain the fundamental solution $Z$ of \eqref{Eq:Intro1}-\eqref{Eq:Intro2}. To this end, we use the subordination principle of operator families in the sense of Bochner and in the sense of Pr\"uss to derive a variation of parameters formula for \eqref{Eq:Intro1}-\eqref{Eq:Intro2}. We also point out that how our approach can be extended to others pseudo-differential operators. In Section \ref{SecL2} we obtain optimal decay in the $L_2$-norm, we illustrate our result with several examples. The decay estimate for $L_r$ is obtained in Section \ref{S:Lr}, here we use tools from the harmonic analysis and we make a careful estimation of the Fourier symbol $\widetilde{Z}(t,\xi)$ of $Z$ with respect to the spatial variable. We derive Fourier multipliers by means of the Mihlin's condition to obtain our main results, which are splitted as follows: for $f=0$ the decay estimate of $u$ is established in Theorem \ref{Theo:Lr:Est:u0} and for the gradient of $u$ in Theorem \ref{Theo:Grad:Sol}. In case of $u_0=0$ the corresponding results are Theorem \ref{Theo:Lr:Est:f} and Theorem \ref{Theo:Grad:Sol} respectively. Finally, in Section \ref {S:Ex} we illustrate in Corollary \ref{Theo:Ex:decay}, Corollary \ref{Theo:Ex:decay:2}  and Corollary \ref{Theo:Ex:decay:3} the different kinds of decay, which are e.g., exponential, algebraic and logarithmic decay.

\medbreak

\section{Relaxation functions $s_{\mu}$ and $r_{\mu}$}\label{relaxation}
We first collect some properties of kernels of type $(\mathcal{PC})$. Let $(k,\ell)\in (\mathcal{PC})$.
For $\mu\in \mathbb{R}$ define the kernels $s_\mu, r_\mu \in L_{1,loc}(\mathbb{R}_+)$ via the scalar Volterra equations
\begin{equation}\label{Func:s}
s_\mu(t)+\mu(\ell\ast s_\mu)(t)  = 1,\quad t>0,
\end{equation}
\noindent and 
\begin{equation}\label{Func:r}
r_\mu(t)+\mu(\ell\ast r_\mu)(t)  = \ell(t),\quad t>0.
\end{equation}
Both $s_\mu$ and $r_\mu$ are nonnegative for all $\mu\in \mathbb{R}$. For $\mu\ge 0$, this is a consequence of the complete positivity of $\ell$ (see \cite[Theorem 2.2]{Cle-Noh-1981} or \cite[Proposition 4.5]{Pru-1993}). If $\mu<0$, this can be seen, e.g.\ by a simple fixed point argument in the space of nonnegative $L_1((0,T))$-functions with arbitrary $T>0$ and an appropriate norm. Moreover, $s_\mu\in H^1_{1,\,loc}([0,\infty))$ for all $\mu\in \mathbb{R}$, and if $\mu\ge 0$, then the function $s_\mu$ is nonincreasing.

Convolving \eqref{Func:r} with $k$, and using that $(k,\ell)\in(\mathcal{PC})$, it follows that $s_\mu=k\ast r_\mu$, by uniqueness. Further, we see that
\begin{equation}\label{1*r=1-s}
\mu(1\ast r_\mu)(t)=1-(k\ast r_\mu)(t)=1-s_\mu(t),\quad t>0,
\end{equation}
which shows that for $\mu>0$ the function $r_\mu$ is integrable on $\mathbb{R}_+$. 

The Laplace transform of a function $f$ is denoted by $\widehat{f}(\lambda)$. We point out that the condition $(k,\ell)\in(\mathcal{PC})$ implies that the Laplace transform of the functions $s_\mu$ and $r_\mu$, are well defined and they are given by
\begin{equation}\label{Lap:s:r}
\widehat{s_\mu}(\lambda)=\dfrac{1}{\lambda(1+\mu\widehat{\ell}(\lambda))},\,  \text{Re} \lambda >0 \text{\ \ and\ \ } \widehat{r_\mu}(\lambda)=\dfrac{\widehat{\ell}(\lambda)}{1+\mu\widehat{\ell}(\lambda)},\, \text{Re} \lambda >0.
\end{equation}

Further, in \cite[Lemma 6.1]{Ver-Zach-2015} the authors have proved that for any $\mu\ge 0$ there holds
\begin{equation}\label{Est:s}
\dfrac{1}{1+\mu k(t)^{-1}}\le s_\mu(t)\le \dfrac{1}{1+\mu(1*\ell)(t)},\text{\ for all\ \ } t\geq 0.
\end{equation}
 
It is well known (see \cite{Cle-Noh-1979}) that if $\ell\in L_{1,loc}(\mathbb{R}_+)$, positive, nonincreasing, then for all $\mu\ge 0$, there exists $r_\mu$ solving the equation \eqref{Func:r}. Moreover we have that $0\le r_\mu(t)\le \ell(t)$ for all $t>0$ and 
\[
\int_0^\infty r_\mu(t)dt=\widehat{r}_\mu(0^+)=\dfrac{\widehat{\ell}(0^+)}{1+\mu\widehat{\ell}(0^+)}\le \dfrac{1}{\mu}.
\]

Furthermore, it follows from \cite[Thereom 3.1]{Gri-Lon-Sta-1990} that $|r_\mu|_1<\frac{1}{\mu}$ if and only if $\ell\in L_{1}(\mathbb{R}_+)$. For more properties about the function $r_\mu$ see \cite[Chapter 5 and Chapter 6]{Gri-Lon-Sta-1990} or \cite[Section 2]{Cle-Noh-1979}.


\section{Variation of parameters formula and subordination principle}\label{Section:Main}

In this section we use the subordination principle of operator families in the sense of Bochner (see e.g. \cite[Chapter 4.3 and Chapter 4.4]{Boch-1955}  and \cite[Section 4.3]{Jac-1996}) and in the sense of Pr\"uss (see \cite{Cle-Noh-1981} and \cite[Chapter 4]{Pru-1993}) to derive a variation of parameters formula for 
\begin{align}
\partial_t(k*[u(t,x)-u_0(x)])+(-\Delta)^{\rho/2} u(t,x)&=f(t,x),\quad t > 0, x\in \mathbb{R}^d,\label{Equation:1}\\
u(0,x)&=u_0(x),\quad x\in\mathbb{R}^d.\label{Equation:2}
\end{align}
With this end, we introduce the Fourier transform of $v\in \mathcal{S}(\mathbb{R}^d)$ by
\[
\widetilde{v}(\xi)=\mathcal{F}(v)(\xi)=(2\pi)^{-d/2}\int_{\mathbb{R}^d}e^{-ix\cdot \xi}v(x)dx, 
\]

\noindent extended as usual to $\mathcal{S}'(\mathbb{R}^d)$. Further, the inverse Fourier transform of $v$, is described by
\[
\mathcal{F}^{-1}(u)(\xi)=\mathcal{F}(u)(-\xi).
\]

We use the definition of the fractional Laplacian operator via Fourier transform, that is,
\begin{equation}\label{Frac:Lap:1}
(-\Delta)^{\rho/2} u(x)=\mathcal{F}^{-1}(|\xi|^\rho \mathcal{F}(u)(\xi))(x).
\end{equation}

We point out that this definition does not impose any restriction on the values $\rho>0$.


\bigskip

\subsection*{Subordination in the sense of Bochner} The following results Theorem \ref{Theo:Conv:Semi} and Theorem \ref{Lem:Sub:Boch} (below) are collected from \cite[Chapter 3 and Chapter 4]{Jac-1996} respectively, and they are well-known from the probability theory and the semigroup theory. 


\begin{definition} A family $(\eta_t)_{t\ge 0}$ of bounded Borel measures on $\mathbb{R}^d$ is called a convolution semigroup on $\mathbb{R}^d$ if the following conditions are fulfilled:
\begin{enumerate}[$(i)$]
\item $\eta_t(\mathbb{R}^d)\le 1$.
\item $\eta_{t+s}=\eta_{t}\star\eta_s$ for all $s,t\ge 0$, and $\eta_0=\delta_0$.
\item $\eta_t\to\delta_0$ vaguely as $t\to 0$, that is 
$$\lim_{t\to 0}\int_{\mathbb{R}^d}\phi(x)\eta_t(dx)=\int_{\mathbb{R}^d}\phi(x)\delta_0(dx),\quad\text{\ for all \ } \phi\in C_0(\mathbb{R}^d).$$
\end{enumerate}
\end{definition}


\begin{definition}
A function $\mu\colon \mathbb{R}^d\to \mathbb{C}$ is called negative definite if $\mu(0)\ge 0$ and the mapping $\xi\mapsto (2\pi)^{-d/2}e^{-t\mu(\xi)}$ is positive definite for all $t\ge 0$.
\end{definition}


\begin{theorem}\label{Theo:Conv:Semi} 
For any convolution semigroup $(\eta_t)_{t\ge 0}$ there exists a uniquely determined continuous negative definite function $\mu\colon \mathbb{R}^d\to \mathbb{C}$ such that 
\begin{equation}\label{Conv:Semi}
\widetilde{\eta_t}(\xi)=(2\pi)^{-d/2}e^{-t\mu(\xi)},\quad t\ge 0, \quad \xi\in\mathbb{R}^d.\end{equation}
Reciprocally, for any continuous negative definite function $\mu\colon \mathbb{R}^d\to \mathbb{C}$, there exists a unique convolution semigroup $(\eta_t)_{t\ge 0}$ satisfying the identity \eqref{Conv:Semi}.
\end{theorem}


\begin{theorem}\label{Lem:Sub:Boch} For any Bernstein function $g\colon \mathbb{R}\to\mathbb{R}$ and any  continuous negative definite function $\mu\colon \mathbb{R}^d\to \mathbb{C}$, the function $(g\circ\mu)$ is continuous negative definite as well. Consequently, there exists a unique convolution semigroup $(\eta_t^g)_{t\ge 0}$ such that 
\[
\widetilde{\eta_t^g}(\xi)=(2\pi)^{-d/2}e^{-tg(\mu(\xi))},\quad t\ge 0, \quad \xi\in\mathbb{R}^d.
\]
\end{theorem}


Since $\mu\colon \mathbb{R}^d\to \mathbb{C}$ is a continuous negative definite function, it follows from Theorem \ref{Theo:Conv:Semi} that there exists a convolution semigroup $(\eta_t)_{t\ge 0}$ associated to the symbol $\mu(\xi)$. In the situation of Theorem \ref{Lem:Sub:Boch}, the new convolution semigroup $(\eta_t^g)_{t\ge 0}$ is called convolution semigroup subordinate to $(\eta_t)_{t\ge 0}$. 


\medbreak

Next, we show the existence of a convolution semigroup associated to $\mu(\xi)=|\xi|^\rho$ with $\rho>0$. We first note that the symbol of $(-\Delta)$ is $|\xi|^2$ and the mapping $\xi\mapsto |\xi|^2$ is a continuous negative definite function. By Theorem \ref{Theo:Conv:Semi} there exists a unique convolution semigroup $(\eta_t)_{t\ge 0}$ associated to $|\xi|^2$. Further the operator
\[
T_t  v(x):=\int_{\mathbb{R}^d} v(x-y)\eta_t(dy), \; t\geq 0,
\]
defines a contraction $C_0$-semigroup on $L_p(\mathbb{R}^d)$, see e.g., \cite[Theorem 3.6.16]{Jac-1996}. By a direct computation, we obtain that for all $n\in\mathbb{N}$ the function $\xi\mapsto |\xi|^{2n}$, is continuous negative definite. Further, if $s\in(0,1]$, then $g(\lambda)=\lambda^s$ for $\lambda>0$, is a Bernstein function. 

We note that $|\xi|^{\rho} = g(|\xi|^{2n})$, with $s=\rho/2n$. We can choose $n\in\mathbb{N}$ such that $\rho/2n\in(0,1]$. Therefore, it follows from Theorem \ref{Lem:Sub:Boch} that $|\xi|^{\rho}$ is continuous negative definite and there exists a convolution semigroup $(\eta^{\rho}_t)_{t\ge 0}$ associated to 
$|\xi|^{\rho}$. Moreover,  the operator
\begin{equation}\label{Semi:Subor}
T^g_tv(x):=\int_{\mathbb{R}^d} v(x-y)\eta^{\rho}_t(dy), \; t\geq 0,
\end{equation}
defines a contraction $C_0$-semigroup on $L_p(\mathbb{R}^d)$. On the other hand, the operator $T^g_t$ can also obtained from the semigroup $T_t $ as follows
\[
T_t^g v = \int_0^{\infty} T_s v \eta^*_t(ds),
\]
where $(\eta^*_t)_{t\geq 0}$ is the convolution semigroup on $\mathbb{R}$ supported on $[0,\infty)$ asociated with the Bernstein function $g(\lambda)=\lambda^{s}$, $\lambda>0$ and  $s\in(0,1)$. The existence of such convolution semigroup is guaranteed by \cite[Theorem 4.3.1]{Jac-1996}. The semigroup $(T_t^g)_{t\geq 0}$ is called subordinate in the sense of Bochner to $(T_t)_{t\geq 0}$ with respect to $(\eta_t)_{t\geq 0}$ or equivalently with respect to $|\xi|^{2}$.  


The convolution semigroup $(\eta^{\rho}_t)_{t\ge 0}$ is well-known in case of $\rho=2$, that is the Gaussian semigroup 
\[
\eta^2_t(x)=(4\pi t)^{\frac{-d}{2}}\exp\left(\frac{-|x|^2}{4t}\right), \quad t>0, \quad x\in\mathbb{R}^d.
\]
For $\rho\in(0,2)$, the corresponding convolution semigroup $(\eta_t^\rho)_{t\ge 0}$ is called symmetric stable semigroup, see \cite[Example 3.9.17]{Jac-1996}.  In general $(\eta^\rho_t)_{t\ge 0}$ is not known in an explicit form. However, it has been proved in \cite[Lemma 3]{Bog-Jak-2007} that 
\[
\eta^\rho_t(x)\asymp \dfrac{t}{(t^{2/\rho}+|x|^2)^{(d+\rho)/2}}, \quad t>0\quad x\in\mathbb{R}^d,
\]
where $\asymp$ means that the ratio bounded by a constant factor from above and below. In the case $\rho=1$, we have the explicit formula
\[
\eta^1_t(x)=\dfrac{\Gamma((d+1)/2)}{\pi^{(d+1)/2}}\dfrac{t}{(t^2+|x|^2)^{(d+1)/2}}, \quad t>0\quad x\in\mathbb{R}^d.
\]
In particular, if $d=1$, this semigroup is known as Cauchy semigroup, (see \cite[Example 3.9.17]{Jac-1996}).

\bigskip


\subsection*{Fundamental Solution $Z$} By definition the fundamental solution $Z(t,x)$ of the problem (\ref{Equation:1})-(\ref{Equation:2}) is the distributional solution of
\begin{equation}\label{Fund:sol}
\partial_t(k*[Z-Z_0])+(-\Delta)^{\rho/2}Z=0,\quad t > 0,\  x\in \mathbb{R}^d,\quad Z_{{|}_{t=0}}=Z_0:=\delta_0,\  x\in \mathbb{R}^d,
\end{equation}
\noindent where $\delta_0$ stands for the Dirac delta distribution.  Applying Fourier transform with respect to $x$, we see that $\widetilde{Z}$ solves the equation
\begin{align*}
\partial_t\bigl(k*[\widetilde{Z}(\cdot,\xi)-1]\bigr)(t)+|\xi|^\rho\widetilde{Z}(t,\xi)&=0,\quad t > 0,\quad \xi\in \mathbb{R}^d,\\
\widetilde{Z}(0,\xi)&=1,\quad\xi\in \mathbb{R}^d.
\end{align*}

\noindent Since $\widetilde{Z}(0,\xi)=1$ and $(k,\ell)\in(\mathcal{PC})$, the preceding equation is equivalent to the Volterra equation
\begin{equation}\label{Eq:Volterra}
\widetilde{Z}(t,\xi)+|\xi|^\rho\bigl(\ell*\widetilde{Z}(\cdot,\xi)\bigr)(t)=1,\quad t > 0,\quad \xi\in \mathbb{R}^d.
\end{equation}

Therefore the solution of the equation (\ref{Eq:Volterra}) is given by $\widetilde{Z}(t,\xi)=s(t,|\xi|^\rho),$ for $t\geq 0$ and $\xi\in \mathbb{R}^d,$ where $s(t,|\xi|^\rho)$ is the function defined in \eqref{Func:s}.


\medskip

In what follows we construct a function $Z$ such that $\mathcal{F}_{_{x\to\xi}}(Z(t,x))=s(t,|\xi|^\rho)$. Recalling that $(k,\ell)\in (\mathcal{PC})$ implies that $\ell$ is a completely positive function \cite[Proposition 4.5]{Pru-1993} , which in turn implies that the function
\[
\varphi(\lambda)=\lambda\widehat{k}(\lambda)=\frac{1}{\widehat{\ell}(\lambda)},\quad \lambda>0,
\]
\noindent is a Bernstein function. Further, for $\tau\ge 0$ we define $\psi_\tau(\lambda)=\exp(-\tau\varphi(\lambda))$ and using again \cite[Proposition 4.5]{Pru-1993} we conclude that $\psi_\tau$ is completely monotone. By Bernstein's theorem (see \cite[Theorem 1.4]{Sch-Son-Von-2010}) there exists a unique non-decreasing function $w(\cdot,\tau)\in BV(\mathbb{R}_+)$ normalized by $w(0,\tau)=0$ and left-continuous such that
\[
\widehat{w}(\lambda,\tau)=\int_0^\infty e^{-\lambda \sigma}w(\sigma,\tau)d\sigma=\frac{\psi_\tau(\lambda)}{\lambda},\quad \lambda>0.
\]

The function $w(t,\tau)$ is the so-called {\it propagation function} associated with the
completely positive kernel $\ell$, see \cite[Section 4.5]{Pru-1993}. Some properties of
$w$ can be found in \cite[Proposition 4.9]{Pru-1993}. For example, $w(\cdot, \cdot)$ is Borel measurable on $\mathbb{R}_+\times \mathbb{R}_+$, the function $w(t, \cdot)$ is non-increasing and right-continuous on $\mathbb{R}_+$, and $w(t, 0) = w(t, 0+) = 1$ as well as $w(t,\infty) = 0$ for all $t > 0$. For our purposes, the most important property of $w$ is the following relation with $s(t,\mu)$, 
\begin{equation}\label{Relax:Propag}
s(t,\mu)=-\int_0^\infty e^{-\tau\mu }  w(t,d\tau), \quad t>0,\ \mu\ge 0.
\end{equation}

We now define
\begin{equation}\label{Fund:Sol}
Z(t,x)=-\int^\infty_0 \eta^\rho_\tau(x) w(t,d\tau)\quad t>0,\ x\in\mathbb{R}^d,
\end{equation}

\noindent where  $(\eta^\rho_t)_{t\ge 0}$ is the convolution semigroup associated to $|\xi|^{\rho}$. By Theorem \ref{Lem:Sub:Boch}, we have that $\mathcal{F}(\eta^\rho_t)(\xi)=\exp(-t|\xi|^\rho)$ for all $t>0$. By (\ref{Relax:Propag}) we have
\[
\widetilde{Z}(t,\xi)=-\int^\infty_0 \widetilde{\eta^\rho_\tau}(\xi) w(t, d\tau)=-\int^\infty_0 e^{-\tau|\xi|^\rho}w(t,d\tau)=s(t,|\xi|^\rho),\quad t>0,
\]

\noindent and $Z$ defined in \eqref{Fund:Sol} is the fundamental solution of \eqref{Fund:sol}.

\medskip

\subsection*{Critical case.} We point out that, in general $Z(t,\cdot)$ does not belong to $L_p(\mathbb{R}^d)$ for arbitrary $p\ge 1$. Indeed, if we assume that $Z(t,\cdot)\in L_p(\mathbb{R}^d)$ for $p\le 2$, then the Haussdorf-Young's inequality  implies that
\[
|\widetilde{Z}(t,\cdot)|_{p'}\le |Z(t,\cdot)|_p<\infty,
\]
where $p'$ is the conjugate exponent of $p$. By definition we have $\widetilde{Z}(t,\xi)=s(t,|\xi|^\rho)$, therefore using the estimates for $s(t,|\xi|^\rho)$ given in (\ref{Est:s}), we have that
\[
\int_{\mathbb{R}^d}\dfrac{d\xi}{(1+|\xi|^\rho k(t)^{-1})^{p'}}<\infty,
\]
which in turns implies (by changing to polar coordinates) that
\[
\int_0^\infty \dfrac{r^{d-1}}{(1+r^\rho)^{p'}}dr<\infty.
\]
Hence $\rho p'-(d-1)>1,$ which is equivalent to $p<\dfrac{d}{d-\rho}$. 

Further, if $p=\frac{d}{d-\rho}$ and $d>2\rho$, then $Z(t,\cdot)\notin L_p(\mathbb{R}^d)$. Indeed, if $d>2\rho$, then $1<\frac{d}{d-\rho}<2$. The conclusion follows from the Haussdorf-Young's inequality. For the sake of brevity, we denote 
\begin{equation}\label{sigma:1}
\sigma_1(\rho,d)=\begin{cases}\dfrac{d}{d-\rho},& d>\rho,\\ \infty, & \text{otherwise}.\end{cases}
\end{equation}

\medskip

\subsection*{Variation of parameters formula for Volterra equations} Defining the operator family $\{S(t)\}_{t\ge 0}$ by
\begin{equation}\label{Sub:Pr}
S(t)v(x)= -\int_{0}^{\infty} T^g_{\tau} v(x) \, w(t,d\tau),
\end{equation}
we obtain
\[
S(t)v(x)= -\int_0^{\infty} (v\star \eta_{\tau}^{\rho})(x)w(t,d\tau)  =(Z(t, \cdot)\star v)(x),
\]
by (\ref{Semi:Subor}) and \eqref{Fund:Sol}. Moreover,  $\{S(t)\}_{t\ge 0}$ is a resolvent family on $L_p(\mathbb{R}^d)$ by \cite[Corollary 4.5]{Pru-1993}. The family defined by \eqref{Sub:Pr} is the so-called {\it subordinated resolvent family} in the sense of Pr\"uss, (see \cite[Chapter 4]{Pru-1993}), to the semigroup $\{T_t^g\}_{t\ge 0}$. 

By condition $(\mathcal{PC})$ we note that equation \eqref{Equation:1} is equivalent to the Volterra equation
\begin{equation}\label{Eq:Volterra:2}
u + \ell \ast (-\Delta)^{\rho/2} u = u_0 + \ell \ast f,
\end{equation}
 Hence by the {\it variation of parameters formula} for Volterra equations (see \cite[Proposition 1.2]{Pru-1993}) we have that the mild solution of \eqref{Equation:1}-\eqref{Equation:2} is given by
\begin{equation}\label{Variation}
u(t,\cdot)=\frac{d}{dt} \bigl(S\ast (u_0 + \ell \ast f)\bigr)(t) = \bigl(Z(t,\cdot)\star u_0\bigr)(\cdot) + \int_0^t Y(t-s,\cdot)\star f(s,\cdot)ds,
\end{equation}
\noindent where $Y(t,x)$ solves the Volterra equation of the first kind $(Y(\cdot,x)\ast k)(t) =Z(t,x)$. Since $(k,\ell)\in(\mathcal{PC})$ we have that $\widetilde{Y}(t,\xi)=\partial_t\bigr(\ell*\widetilde{Z}(\cdot,\xi)\bigl)(t) = r(t,|\xi|^\rho)$ for all $t > 0$, where the function $r(t,|\xi|^\rho)$ is defined as solution of  \eqref{Func:r}.

We remark that a function described by the variation parameters formula is not necessarily a strong solution of the problem \eqref{Equation:1}-\eqref{Equation:2}. This depends on the properties of $u_0$ and $f$, see \cite[Proposition 1.3]{Pru-1993}.

We also observe that if $\ell(t)=1$ for $t\geq 0$ in equation \eqref{Eq:Volterra:2}, then
\[
Z(t,x)=Y(t,x)=\eta^g_t(x), \; t>0, \; x\in \mathbb{R}^d.
\] 
Moreover, the function $u$ given by \eqref{Variation} is the unique solution of the local in time equation
\begin{align*}
u_t+(-\Delta)^{\rho/2} u &= f,\quad t > 0,\; x\in \mathbb{R}^d,\\
  u_0(x) &= u(0,x),\quad  x\in \mathbb{R}^d.
\end{align*}
In this case the mild solution can be written as
\[
u(t,x) = T^g_t u_0(x) + \int_0^t T_s^g f(t-s,x)ds,
\]
where $(T^g_t)_{t\geq 0}$ is a contraction $C_0$-semigroup defined in \eqref{Semi:Subor}.

Now, considering $k(t)=g_{1-\alpha}(t)$ with $\alpha\in(0,1)$ and $t>0$, problem \eqref{Equation:1}-\eqref{Equation:2} takes the form of the equation studied in  \cite{Kem-Sil-Zach-2017} and \cite{Kim-Lim-2016}. Further, in this case we have $\ell(t)=g_\alpha(t)$, and explicit formulas for the relaxation function and the integrated relaxation function, that is
\[
s(t,|\xi|^\rho)=(2\pi)^{-d/2}E_{\alpha,1}(-t^\alpha|\xi|^\rho),\  t\ge 0,
\]
and
\[
r(t,|\xi|^\rho)=(2\pi)^{-d/2}t^{\alpha-1}E_{\alpha,\alpha}(-t^\alpha|\xi|^\rho),\ t\ge 0,
\]
\noindent where the function $E_{\alpha,\beta}$ corresponds to the generalized Mittag-Leffler function, see e.g. \cite[Apendix E.2]{Mai-2010}. However, for general kernels $(k,\ell)\in(\mathcal{PC})$, there is no explicit formula for $s_\mu$ and $r_\mu$ in the literature. Hence, the analysis of the properties of the solution of the problem is more difficult. In a section below we illustrate how our results can be applied for different examples $(k,\ell)\in(\mathcal{PC})$ finding the decay rates of the solution to the corresponding non-local problem \eqref{Equation:1}-\eqref{Equation:2}.

\medbreak

We point that the procedure described in this section can be extended to a linear operator $\mathcal{L}$ on $L_p(\mathbb{R}^d)$ whose symbol $\mu :\mathbb{R}^d \to \mathbb{C}$ is given by the L\'evy-Khinchin formula (cf. \cite[Theorem 3.7.7]{Jac-1996}), that is
\[
\mu(\xi) = c + i b\cdot \xi + q(\xi) +\int_{\mathbb{R}^d\setminus \{0\}} \left(1-e^{-ix\cdot \xi} - \frac{ix\cdot \xi}{1+|x|^2}\right)\frac{1+|x|^2}{|x|^2} \nu(dx).
\]
Indeed, in this case the constant $c\geq 0$, the vector $b\in \mathbb{R}^d$, the symmetric positive semidefinite quadratic form $q$ on $\mathbb{R}^d$ and the finite measure $\nu$ on $\mathbb{R}^d\setminus \{0\}$ are uniquely determined by $\mu(\cdot)$ see \cite[Theorem 3.7.8]{Jac-1996}. Moreover, the function $\mu$ is a continuous negative definite (cf. \cite[Theorem 3.7.7 and Theorem 3.7.8]{Jac-1996}), thus by Theorem \ref{Theo:Conv:Semi} there exists a unique convolution semigroup $(\eta_t)_{t\geq 0}$  on $\mathbb{R}^d$ such that (\ref{Conv:Semi}) holds, hence $Z$ can be defined by
\[
Z(t,x)=-\int_0^{\infty} \eta_{\tau}(x)w(t,d\tau).
\]
Furthermore, by Theorem \ref{Lem:Sub:Boch} any Bernstein function $g$ composed with any continuous negative definite function $\mu$ is continuous and negative definite function and consequently there exists a unique convolution semigroup $(\eta_t^g)_{t\geq 0}$. Hence, the fundamental solution $Z$ for the problem  
\[
\partial_t(k\ast [Z-Z_0])(t) + g(\mathcal{L})Z =0, \; t>0,\; x\in \mathbb{R}^d,\; Z_{{|}_{t=0}}=Z_0:=\delta_0,\  x\in \mathbb{R}^d,
\]
with $g(\mathcal{L}) u(x) :=\mathcal{F}^{-1}((g\circ\mu)(\xi)) \mathcal{F}(u)(\xi))(x)$, can be defined as the subordinated kernel to the convolution semigroup $(\eta_t^g)_{t\geq 0}$ as in \eqref{Fund:Sol}.

\medbreak

The following result shows an interesting property of the kernels $Z$ and $Y$, which is well-known in case of $\rho=2$ and $k=g_{1-\alpha}$, see \cite{Eid-Koch-2004}, and for $\rho=2$ and $k$ as the example \ref{k:dist:order}, see \cite[Lemma 4.2 and Proposition 5.2]{Koch-2008}.

\begin{proposition} 
Let $(k,\ell)\in(\mathcal{PC})$ and $d\in \mathbb{N}$. Let $Z$ be the kernel given by \eqref{Fund:Sol} and $Y$ the solution of the Volterra equation of the first kind $(Y(\cdot,x)\ast k)(t)=Z(t,x)$. If $\rho=2$, or $\rho=1$ and $d=1$, then the following assertion holds.
\[
\int_{\mathbb{R}^d} Z(t,x) dx = 1, \text{ and } \int_{\mathbb{R}^d}Y(t,x)dx = \ell(t),\; t>0.
\]
 \end{proposition}
\begin{proof}
We note that 
\[
\int_{\mathbb{R}^d}\eta^2_t(x)dx=1, \text{ and } \int_{\mathbb{R}^d}\eta^1_t(x)dx=t^{1-d},\; t>0.
\]
Taking the Laplace transform to $Z$ in the temporal variable, we have
\begin{align*}
\widehat{Z}(\lambda,x)  & =-\int_0^{\infty}e^{-\lambda t}\int_0^{\infty} \eta^{\rho}_{\tau}(x)w(t,d\tau)dt = -\int_0^{\infty} \eta^{\rho}_{\tau}(x) \frac{1}{\lambda}\partial_{\tau}\psi_{\tau}(\lambda)d\tau\\
& = \int_0^{\infty} \eta^{\rho}_{\tau}(x) \hat{k}(\lambda) e^{-\tau \lambda \hat{k}(\lambda)}d\tau,
\end{align*}
by \cite[Chapter 4, Proposition 4.9]{Pru-1993}. Integrating the last equality over the spatial variable and by Fubini's theorem we obtain
\begin{equation}\label{Eq:Laplace-Z}
\int_{\mathbb{R}^d}\widehat{Z}(\lambda,x)   = \int_0^{\infty} \int_{\mathbb{R}^d}\eta^{\rho}_{\tau}(x) dx \, \hat{k}(\lambda)e^{-\tau \lambda \hat{k}(\lambda)} d\tau= \frac{1}{\lambda},
\end{equation}
for $\rho=2$, or for $\rho=1$ and $d=1$.

Since $\widehat{Z}(\lambda,x) = \widehat{Y}(\lambda,x) \widehat{k}(\lambda)$, we have 
\[
\int_{\mathbb{R}^d} \widehat{Y}(\lambda,x)dx = \frac{1}{\lambda \widehat{k}(\lambda)} = \widehat{\ell}(\lambda)
\]
by (\ref{Eq:Laplace-Z}) and the condition $(\mathcal{PC})$. Hence, the claim follows by the inversion of the Laplace transform.
\end{proof}


\section{Optimal $L_2$-decay for mild solutions }\label{SecL2} 

In this section, we derive $L_2$-estimates decay for the solutions of \eqref{Equation:1}-\eqref{Equation:2}. Our results generalize those obtained in \cite[Section 4]{Kem-Sil-Ver-Zach-2016} and \cite[Section 6]{Kem-Sil-Zach-2017}. We consider $f\equiv 0$. In this case, the solution of our problem is given by
\begin{equation}\label{Variation:2}
u(t,x)=\int_{\mathbb{R}^d}Z(t,x-y)u_0(y)dy=(Z(t,\cdot)\star u_0)(x).
\end{equation}


\begin{theorem}\label{Theo:Up:L2} Let $d\in \mathbb{N}$ such that $d\neq 2\rho$ and $u_0\in L_1(\mathbb{R}^d)\cap L_2(\mathbb{R}^d)$. If $(k,\ell)\in(\mathcal{PC})$ and $u$ is described by the formula (\ref{Variation:2}) then
\[
|u(t,\cdot)|_2\lesssim (1*\ell)(t)^{-\min\left\{1,\frac{d}{2\rho}\right\}},\  t>1.
\]
Here the notation $f(t)\lesssim g(t)$, $t>t_*$, means that there exists $C>0$ such that $f(t)\le C g(t)$ for $t>t_*$.  

\end{theorem}

\begin{proof} Initially we consider $d<2\rho$. Since $u_0\in L_1(\mathbb{R}^d)\cap L_2(\mathbb{R}^d)$ we have that $\widetilde{u_0}\in C_0(\mathbb{R}^d)\cap L_2(\mathbb{R}^d)$ and by Plancherel's theorem we have

\begin{align*}
(2\pi)^d |u(t,\cdot)|_2^2&=|\widetilde{u}(t,\cdot)|_2^2=\int_{\mathbb{R}^d} \widetilde{Z}(t,\xi)^2|\widetilde{u_0}(\xi)|_2^2d\xi\le |\widetilde{u_0}|_{\infty}^2\int_{\mathbb{R}^d} s(t,|\xi|^\rho)^2 d\xi \\
&\le C|u_0|^2_1\int_{\mathbb{R}^d} \dfrac{d\xi}{(1+|\xi|^\rho(1*\ell)(t))^2}= C|u_0|^2_1\int_{0}^\infty \dfrac{r^{d-1}dr}{(1+ r^\rho(1*\ell)(t))^2}\\
&\le \dfrac{C|u_0|^2_1}{(1*\ell(t))^{d/\rho}}\int_{0}^\infty \dfrac{r^{d-1}dr}{(1+r^\rho)^2}= K(1*\ell)(t)^{-d/\rho}, 
\end{align*}

\noindent for some $K>0$. It is clear that if $d\ge 2\rho$, this argument is not valid because the last integral is divergent. 

For $d>2\rho$, by interpolation arguments we note that $u_0\in L_p(\mathbb{R}^d)$ where $p= \frac{2d}{2\rho+d}$. The Hardy-Littlewood inequality on fractional integration (see \cite[Theorem 6.1.3]{Gra-2004})  implies that $(-\Delta)^{-\rho/2}u_0\in L_2(\mathbb{R}^d)$. Therefore 
\begin{align*}
(2\pi)^d |u(t,\cdot)|_2^2&=|\widetilde{u}(t,\cdot)|_2^2=\int_{\mathbb{R}^d} \widetilde{Z}(t,\xi)^2|\widetilde{u_0}(\xi)|_2^2d\xi= \int_{\mathbb{R}^d} |\xi|^{2\rho}s(t,|\xi|^\rho)^2 \bigl||\xi|^{-\rho}|\widetilde{u_0}(\xi)|\bigr|^2 d\xi\\
&\le \dfrac{1}{(1*\ell)(t)^2}\int_{\mathbb{R}^d} \dfrac{|\xi|^{2\rho}(1*\ell)(t)^2}{(1+|\xi|^{2\rho}(1*\ell)(t)^2)^2}\bigl||\xi|^{-\rho}|\widetilde{u_0}(\xi)|\bigr|^2 d\xi\\
&\le \dfrac{1}{(1*\ell)(t)^2}\int_{\mathbb{R}^d} |\xi|^{-\rho}|\widetilde{u_0}(\xi)|^2 d\xi=\dfrac{(2\pi)^d|(-\Delta)^{-\rho/2}u_0|^2_2}{(1*\ell)(t)^2}.
\end{align*}

Consequently, $|u(t,\cdot)|_2\lesssim (1*\ell)(t)^{-\min\left\{1,\frac{d}{2\rho}\right\}},\ \text{a.a \ } t>1.$
\end{proof}


To see if the estimate obtained in Theorem \ref{Theo:Up:L2} is optimal we establish the following result. 


\begin{theorem}\label{Theo:Low:L2} Let $d\in \mathbb{N}$ and $u_0\in L_1(\mathbb{R}^d)\cap L_2(\mathbb{R}^d)$. Assume that $\widetilde{u_0}(0)\neq 0$. If $(k,\ell)\in(\mathcal{PC})$ and $u$ is described by the formula (\ref{Variation:2}) then

\[
|u(t,\cdot)|_2\gtrsim k(t)^{\min\left\{1,\frac{d}{2\rho}\right\}},\ \text{a.a \ } t>1.
\]
\end{theorem}

\begin{proof} Let $t>0$, $R>0$ and $r\in(0,R]$. By Plancherel's theorem and the
monotonicity property of $s_\mu$ with respect to $\mu>0$ we obtain that
\begin{align*}
(2\pi)^d |u(t,\cdot)|_2^2&=|\widetilde{u}(t,\cdot)|_2^2=\int_{\mathbb{R}^d} \widetilde{Z}(t,\xi)^2|\widetilde{u_0}(\xi)|_2^2d\xi\ge \int_{B_r} \widetilde{Z}(t,\xi)^2|\widetilde{u_0}(\xi)|_2^2d\xi \\
&\ge\int_{B_r} s(t,|\xi|^\rho)^2|\widetilde{u_0}(\xi)|_2^2d\xi\ge s(t,r^\rho)^2\int_{B_r} |\widetilde{u_0}(\xi)|_2^2d\xi\\
&\geqslant r^d s(t,r^\rho)^2 r^{-d}\int_{B_r} |\widetilde{u_0}(\xi)|_2^2d\xi.
\end{align*}

We recall that $u_0\in L_1(\mathbb{R}^d)\cap L_2(\mathbb{R}^d)$ implies that $\widetilde{u_0}\in C_0(\mathbb{R}^d)\cap L_2(\mathbb{R}^d)$. Since $\widetilde{u_0}(0)\neq 0$ we can choose $R$ small enough such that
\begin{equation}\label{Def:R}
\left(r^{-d}\int_{B_r} |\widetilde{u_0}(\xi)|_2^2d\xi\right)>C, \text{\ \ for all \ } r\in(0,R],
\end{equation}
\noindent for some $C>0$. Using $R>0$ as in (\ref{Def:R}) we consider $r=R$ and note that 
\begin{align*}
r^d s(t,r^\rho)^2&=R^d s(t,R^\rho)^2\ge \dfrac{R^d}{(1+R^\rho k(t)^{-1})^2}\\
&\ge \dfrac{R^d k(t)^2}{(k(t)+R^\rho)^2} \ge \dfrac{R^d k(t)^2}{(k(1)+R^\rho)^2}\gtrsim k(t)^2, \ t>1.
\end{align*}

On the other hand, we can choose $r$ as $r(t)=\dfrac{R}{(1+k(t)^{-1})^{1/\rho}}<R$. We note that $k(t)^{-1}r(t)^\rho<R^\rho$ and 
\begin{align*}
r^d s(t,r^\rho)^2&\ge \dfrac{R^d}{(1+k(t)^{-1})^{d/\rho}} \dfrac{1}{(1+r(t)^\rho k(t)^{-1})^2 }\ge \dfrac{R^d}{(1+R^\rho)^2}\dfrac{k(t)^{\frac{d}{\rho}}}{(1+k(t))^{\frac{d}{\rho}}}\\
&\ge \dfrac{R^d}{(1+R^\rho)^2}\dfrac{k(t)^{\frac{d}{\rho}}}{(1+k(1))^{\frac{d}{\rho}}}\gtrsim k(t)^{\frac{d}{\rho}}, \ t>1.
\end{align*}
Therefore
\[
|u(t,\cdot)|_2\gtrsim k(t)^{\min\left\{1,\frac{d}{2\rho}\right\}},\ \text{a.a \ } t>1.
\]
\end{proof}


\begin{remark} According to Theorem \ref{Theo:Up:L2} and Theorem \ref{Theo:Low:L2} the decay rate  of the solution is optimal if there exists $T>0$ such that 
\begin{equation}\label{Optimal}
[(1\ast\ell)(t)]^{-1}\lesssim k(t).
\end{equation}
\end{remark}

We point out that (\ref{Optimal}) is valid for several kernels $(k,\ell)\in(\mathcal{PC})$ as we next illustrate with some examples. To this end, we introduce a version of Karamata-Feller Tauberian theorem, which establish that the asymptotic behaviour of a function $w(t)$ as $t\to \infty$ can be determined, under suitable conditions, by looking at the behaviour of its Laplace transform $\widehat{w}(z)$ as $z\to 0$, and vice versa. See the monograph \cite[Section 5, Chapter XIII]{Fell-1971} for a more general version and proofs.

\medskip

\begin{theorem} \label{Karamata}
Let $L:(0,\infty)\to (0,\infty)$ be a function that is {\em slowly varying at $\infty$}, that is, for every fixed
$x>0$ we have $L(tx)/L(t)\to 1$ as $t\to \infty$. Let $\beta>0$ and $w:(0,\infty)\rightarrow \mathbb{R}$ be a monotone function whose Laplace transform $\widehat{w}(z)$ exists for all $z\in \mathbb{C}_+:=\{\lambda\in \mathbb{C}:\, \mbox{Re}\,\lambda>0\}$. Then
\[
\widehat{w}(z) \sim \,\frac{1}{z^\beta}\,L\left(\frac{1}{z}\right)\quad \mbox{as}\;z\to 0\quad\; \mbox{if and only if} \quad \;
w(t)\sim \frac{t^{\beta-1}}{\Gamma(\beta)}L(t)\quad \mbox{as}\;t\to \infty.
\]
Here the approaches are on the positive real axis and the notation $f(t)\sim g(t)$ as $t\to t_*$ means that there exists a constant $C>0$ such that $\lim_{t\to t_*} f(t)/g(t) =C$. 
\end{theorem}

 \begin{example}\label{Ex:1} {\it The classical time fractional case}. We consider the pair 
\[
(k,\ell)=(g_{1-\alpha},g_{\alpha}),\ \text{\ with \ } \alpha\in(0,1).
\] 

A direct calculation shows that 
\[
(1\ast\ell)(t)=g_{1+\alpha}(t)\sim t^{\alpha}, \text{ \ as \ } t\to\infty.
\]
We note that $k(t)\sim t^{-\alpha}$ as $t\to\infty$. Therefore \eqref{Optimal} holds.
\end{example}

\begin{example}\label{Ex:2}
{\it Sum of two fractional derivatives}. Let $0<\alpha<\beta<1$, and 
\[
k(t)=g_{1-\alpha}(t)+g_{1-\beta}(t),\quad t>0.
\] 
By \cite[Theorem 5.4, Chapter 5]{Gri-Lon-Sta-1990}, there is a positive kernel $\ell \in L_{1,loc}(\mathbb{R}_+)$ such that, $k\ast \ell = 1$ on $(0,\infty)$. In particular $(k,\ell)\in(\mathcal{PC})$. Moreover, the Laplace transform of these kernels are given by
\[
\widehat{k}(\lambda)=\frac{1}{\lambda^{1-\alpha}}+\frac{1}{\lambda^{1-\beta}}, \text{\ and \ }\widehat{\ell}(\lambda)=\dfrac{1}{\lambda^\alpha+\lambda^\beta},\quad \lambda>0.
\]

Since $0<\alpha<\beta<1$, we have that
\[
\widehat{(1\ast\ell)}(\lambda)\sim\dfrac{1}{\lambda^{1+\alpha}},\quad \text{as\ }\lambda\to0.
\]
Therefore, using Theorem \ref{Karamata}, with $L(t)\equiv 1$, we conclude that 
\[
(1\ast \ell)(t)\sim t^{\alpha},\quad \text{as\ }t\to\infty.
\]
Since, $0<\alpha<\beta<1$, we have that $k(t)\sim t^{-\alpha}$ as $t\to \infty$. Hence, (\ref{Optimal}) holds.

These considerations extend trivially to kernels having th form $k(t) =\sum_{j=1}^m \kappa_j g_{1-\alpha_j}(t) $ with $\kappa_j > 0$ and $0 < \alpha_1 < \alpha_2 \cdot \cdots < \alpha_m< 1$.
\end{example}

\medskip


\begin{example}\label{Ex:3}{\it The time-fractional case with weight}. Let $0<\alpha<\beta<1$ and consider 
\[
k(t)=g_{\beta}(t)E_{\alpha,\beta}(-\omega t^{\alpha}),\quad t>0,
\] 
where $\omega>0$ and $E_{\alpha,\beta}$ is the generalized Mittag-Leffler function, see \cite[Appendix E.2]{Mai-2010}. Initially, we analyze the kernel $k$. It is well known that 
\[
\widehat{k}(\lambda)=\dfrac{\lambda^{\alpha-\beta}}{\lambda^\alpha+\omega}=\dfrac{1}{\lambda^{\beta-\alpha}}\dfrac{1}{\lambda^\alpha+\omega}, \text{\ for \ } \lambda>0.
\]

Hence, we have that 
\[
\widehat{k}(\lambda)=\dfrac{1}{\lambda^{\beta-\alpha}} L_1\left(\dfrac{1}{\lambda}\right), \text{\ for \ } \lambda>0,
\]
where $L_1(t)=\frac{t^\alpha}{1+\omega t^\alpha},$ for $t>0$. A direct computation shows that $L_1$ is a slowly varying function at $\infty$. Since, $\alpha<\beta$ we can apply Theorem \ref{Karamata} and conclude that 
\[
k(t)\sim t^{\beta-\alpha-1}, \text{\ as \ }t\to \infty.
\]

Furthermore, it follows from \cite[Theorem 5.4, Chapter 5]{Gri-Lon-Sta-1990} that there is a kernel $\ell \in L_{1,loc}(\mathbb{R}_+)$ such that, $k\ast \ell = 1$ on $(0,\infty)$. Therefore $(k,\ell)\in(\mathcal{PC})$ and    
\[
\widehat{(1\ast\ell)}(\lambda)=\dfrac{1}{\lambda^{2+\alpha-\beta}}(\lambda^\alpha+\omega)=\dfrac{1}{\lambda^{2+\alpha-\beta}}L_2\left(\dfrac{1}{\lambda}\right),\text{\ for \ }\lambda>0,
\]
where $L_2(t)=\frac{1+\omega t^{\alpha}}{t^\alpha}$, for $t>0$. A direct computation shows that $L_2$ is a slowly varying function at $\infty$. Since, $2+\alpha-\beta>0$ we can apply Theorem \ref{Karamata} to conclude that 
\[
(1\ast \ell)(t)\sim t^{1+\alpha-\beta}, \text{\ as \ }t\to \infty.
\]
Therefore, (\ref{Optimal}) holds.

\end{example}


\begin{example}\label{Ex:4}{\it An example of ultraslow diffusion}. Let $n\in\mathbb{N}$ and consider  
\[
k(t)=\int_0^1g_{\alpha}(t)\alpha^n d\alpha, \quad t>0.
\] 
by Fubini's Theorem we note that 
\begin{align*}
\widehat{k}(\lambda)&=\int_0^\infty e^{-\lambda t} k(t)dt=\int_0^\infty \int_0^1 e^{-\lambda t}\ \alpha^n g_{\alpha}(t)\, d\alpha \, dt\\
&=\int_0^1\alpha^n \int_0^\infty e^{-\lambda t}\ g_{\alpha}(t)\, dt\, d\alpha.
\end{align*}
Since $\widehat{g_\alpha}(\lambda)=\lambda^{-\alpha}$, we have that 
\begin{equation}\label{Laplace:k1}
\widehat{k}(\lambda)=\int_0^1 \alpha^n \lambda^{-\alpha}d\alpha=\dfrac{1}{\lambda}\left(\dfrac{\lambda n!}{\log^{n+1}(\lambda)}- \dfrac{n!}{\log^{n+1}(\lambda)}\sum_{m=0}^n\dfrac{\log^{m}(\lambda)}{m!}\right)=\dfrac{1}{\lambda} L_1\left(\dfrac{1}{\lambda}\right),
\end{equation}
where 
\[
L_1(t)=\dfrac{(-1)^{n+1}n!}{\log^{n+1}(t)}\left(\dfrac{1}{t}-\sum_{m=0}^n \dfrac{(-1)^m\log^{m}(t)}{m!}\right).
\]
 The function $L_1$ is slowly varying at $\infty$, and $L_1(t)\sim [\log(t)]^{-1}$ as $t\to \infty$. Therefore it follows from Theorem \ref{Karamata} that
\[
k(t)\sim \dfrac{1}{\log(t)},\quad \text{ as } t\to\infty.
\]

Moreover, using \cite[Proposition 3.1]{Koch-2008} there exists $\ell \in L_{1,loc}(\mathbb{R}^+)$ such that $(k*\ell)=1$. Therefore, we have that $(k,\ell)\in(\mathcal{PC})$ and 
\[
\widehat{(1*\ell)}(\lambda)=\dfrac{1}{\lambda}\cdot\dfrac{\log^{n+1}(\lambda)}{n!\left(\displaystyle\lambda-\sum_{m=0}^n \dfrac{\log^{m}(\lambda)}{m!}\right)}=\frac{1}{\lambda}L_2\left(\frac{1}{\lambda}\right), \text{\ for\ }\lambda>0.
\]
where 
\[
L_2(t)=\dfrac{(-1)^{n+1}t\log^{n+1}(t)}{\displaystyle n!\left(1-\sum_{m=0}^n \dfrac{(-1)^m t\log^m(t)}{m!}\right)},  t>0.
\] 

Since $L_2(t)=[L_1(t)]^{-1}$, we have that $L_2$ is a slowly varying function at $\infty$, and $L_2(t)\sim \log(t)$ as $t\to \infty$. In consequence, we to apply Theorem \ref{Karamata} and conclude that 
\[
(1\ast\ell)(t)\sim \log(t), \text{ \ as\ } t\to \infty.
\]
Therefore, \ref{Optimal} is valid.

\end{example}


\section{$L_r$-estimates of mild solutions}\label{S:Lr}

We start this section recalling the following results about the map $\mu\mapsto s(t,\mu)$. The proof can be found in \cite[Lemma 5.1 and Lemma 5.2]{Kem-Sil-Ver-Zach-2016}. 

\begin{lemma}\label{Lem:aux1}
Assume that $(k,\ell)\in (\mathcal{PC})$. Let $t\ge 0$. Then the map $\mu\mapsto s(t,\mu)$ belongs to $C^{\infty}(\mathbb{R}_+)$ and  
\begin{equation}\label{CM:s}
(-1)^{n}\partial_\mu^n s(t,\mu)\ge 0, \text{ \ for all \ } n\in\mathbb{N}_0, 
\end{equation}
Moreover, we have that 
\begin{equation}\label{Bound:s}
\mu^n |\partial_\mu^n s(t,\mu)|\le 2^n\, n!\, s\left(t,\frac{\mu}{2}\right), \text{ \ for all \ } n\in\mathbb{N}_0.
\end{equation}
\end{lemma}

\begin{lemma}\label{Lem:aux2}
Let $(k,\ell)\in(\mathcal{PC})$ and $t\ge 0$ be fixed. Let $\delta\in(0,1]$ and set $\psi_\delta(\mu)=\mu^\delta s(t,\mu)$, for  $\mu>0$. Then $\psi_\delta\in C^{\infty}((0,\infty))$ and for every $n\in \mathbb{N}$ there is a constant $C(n)$ such that 
\[
\mu^\delta|\psi_\delta^{(n)}(\mu)|[(1\ast\ell)(t)]^{\delta}\le C(n).
\]
\end{lemma}


Next, we obtain analogues results for the mapping $\mu\mapsto r(t,\mu)$.

\begin{lemma}\label{Lem:aux3}
Assume that $(k,\ell)\in (\mathcal{PC})$. Let $t\ge 0$. Then the map $\mu\mapsto r(t,\mu)$ belongs to $C^{\infty}(\mathbb{R}_+)$ and  
\begin{equation}\label{CM:r}
(-1)^{n}\partial_\mu^n r(t,\mu)\ge 0, \text{ \ for all \ } n\in\mathbb{N}_0, 
\end{equation}
Moreover, we have that 
\begin{equation}\label{Bound:r}
\mu^n |\partial_\mu^n r(t,\mu)|\le 2^n\, n!\, r\left(t,\frac{\mu}{2}\right), \text{ \ for all \ } n\in\mathbb{N}_0.
\end{equation}

\end{lemma}

\begin{proof}  We recall that the function $r_\mu$ is the solution of the equation 
\[
r_\mu + \mu(\ell*r_\mu)=\ell.
\]
Since $\mu$ merely appears as coefficient in front of the second term, it is clear that the
dependence of the solution $r_\mu(t)$ on the parameter $\mu$ is $C^{\infty}$. A direct computation shows that differentiating $n$-times with respect to $\mu$, we obtain the following recursive equation
\begin{equation}\label{Eq:Rec:rn}
\partial^n_\mu r_\mu +\mu(\ell*\partial^n_\mu r_\mu)=-n\,(\ell*\partial^{n-1}_\mu r_\mu),\ \text{ for all } n\in\mathbb{N}.
\end{equation}

In order to prove the property \eqref{CM:r}, we note that 
\begin{equation}\label{Eq:No:Rec:rn}
\partial_\mu^n r_\mu=(-1)^{n}n!\, r_\mu^{*(n+1)}, \text{ for all\ } n\in \mathbb{N}_0,
\end{equation}

\noindent where $r_\mu^{*(n+1)}=r_\mu\ * \ r_\mu^{(n)}$, and $r_\mu^{*(1)}=r_\mu$. The equality \eqref{Eq:No:Rec:rn} follows from an inductive argument. Indeed, considering $n=1$ in \eqref{Eq:Rec:rn}, we have that 
\[
\partial_\mu r_\mu+\mu(\ell*\partial_\mu r_\mu)=-(\ell*r_{\mu}).
\]

Taking Laplace transform into both sides of the preceding equation we obtain that 
\[
\widehat{\partial_\mu r_\mu}=-\dfrac{\widehat{\ell}\cdot\widehat{r_\mu}}{1+\mu\widehat{\ell}}=-\widehat{r_\mu}\cdot\widehat{r_\mu},
\]
\noindent by (\ref{Lap:s:r}). This implies that $\partial_\mu r_\mu=-r_\mu*r_\mu$, and \eqref{Eq:No:Rec:rn} is valid for $n=1$. Now, assuming that \eqref{Eq:No:Rec:rn} is true for $n\in\mathbb{N}$ and using Laplace transform, we have 
\[
\widehat{\partial^{n}_\mu r_\mu}=(-1)^{n}\, n!\, \widehat{r_\mu}^{n+1}.
\] 
Using the Laplace transform into both sides of \eqref{Eq:Rec:rn}, we have 
\[
\widehat{\partial^{n+1}_\mu r_\mu}=-\dfrac{(n+1)\cdot \widehat{\ell}\ \widehat{\partial_\mu^n r_\mu}}{1+\mu\widehat{\ell}}=(-1)^{n+1}\ (n+1)!\ (\widehat{r_\mu})^{n+1},
\]
Therefore $\partial_\mu^{n+1} r_\mu=(-1)^{n+1}\, (n+1)!\, r_\mu^{*(n+2)}$, and the induction is finished. Since $r_\mu$ is a non-negative function, it follows from \eqref{Eq:No:Rec:rn} that $r_\mu$ satisfies \eqref{CM:r}.

On the other hand, applying the Taylor's theorem to the map $\mu\mapsto r(t,\mu)$, we obtain 
\begin{equation}\label{Taylor}
r\left(t,\dfrac{\mu}{2}\right)=\sum_{j=0}^n \frac{\partial^{j}_\mu r_\mu(t,\mu)}{j!}\left(-\dfrac{\mu}{2}\right)^{j}+ \dfrac{\partial_\mu^{n+1}r_\mu(t,\eta)}{(n+1)!}\left(-\dfrac{\mu}{2}\right)^{n+1}, \text{ for some } \eta\in\left(\frac{\mu}{2},\mu\right). 
\end{equation}
 
Since $r_\mu(t)$ satisfies \eqref{CM:r}, it follows that every summand of the right hand side of \eqref{Taylor} is non-negative and this implies that 
\[
r\left(t,\frac{\mu}{2}\right)\ge \dfrac{\partial_\mu^j r(t,\mu)}{j!}\left(-\dfrac{\mu}{2}\right)^j, \text{ for all } j<n.
\]
Since $n$ can be chosen arbitrarily, this implies that 
\[
\mu^n| \partial_\mu^n r(t,\mu)|\le n!\ 2^n\  r\left(t,\frac{\mu}{2}\right), \text{ for all } n\in\mathbb{N} \text{ and } \mu\ge0.
\] 
\end{proof}



\medbreak

\begin{lemma} \label{Lem:aux4}
Let $(k,\ell)\in (\mathcal{PC})$ and $t\ge 0$ be fixed. Let $\delta\in[0,1]$ and set $\phi_\delta(\mu)=\mu^\delta r(t,\mu)$ for $\mu> 0$. Then $\phi_\delta\in C^{\infty}((0,\infty))$ and for every $n\in \mathbb{N}$, there exists a constant $C(n)>0$ such that
\begin{equation}\label{phi:delta}
\mu^n | \phi_\delta^{(n)}(\mu)|[(1\ast \ell)(t)]^{\delta}\ell(t)^{-1}\ \le C(n),\quad \mu> 0.
\end{equation}

\end{lemma}


\begin{proof} We begin noting that \eqref{1*r=1-s} together with the upper estimate of $s_\mu$ given by \eqref{Est:s} imply that 
\[
1-\dfrac{1}{1+\mu(1\ast\ell)(t)}\le \mu(1\ast r_\mu)(t),
\]
which in turns implies that  
\[
\dfrac{(1\ast\ell)(t)}{1+\mu(1\ast\ell)(t)}\le (1\ast r_\mu)(t), \quad t>0,\  \mu\ge 0.
\]
Further, since $\ell$ is nonincreasing we have that 
\[
\ell(t)\ge r_\mu(t)+\mu\ell(t)(1\ast r_\mu)(t) \ge r_\mu(t)+\mu\ell(t)\dfrac{(1\ast\ell)(t)}{1+\mu (1\ast \ell)(t)}.
\]
Hence, we have 
\begin{equation}\label{r:up:est}
r_\mu(t)\le \dfrac{\ell(t)}{1+\mu(1\ast\ell)(t)}, \quad t>0,\  \mu\ge 0.
\end{equation}

On the other hand, by Leibniz's rule for  differentiation of a product of two functions we have that  
\[
\mu^n \phi_\delta^{(n)}(\mu)=\sum_{j=0}^n\binom{n}{j} \bigl(\mu^j\partial_\mu^j r(t,\mu)\bigr)\cdot \bigl(\mu^{n-j}\partial_\mu^{n-j}(\mu^{\delta})\bigr). 
\]
It follows from \eqref{Bound:r} that there exists a constant $\tilde{C}(n)$ such that
\begin{align*}
\mu^n |\phi_\delta^{(n)}(\mu)|&\le\tilde{C}(n)\mu^\delta \sum_{k=0}^n\binom{n}{k}  2^k k!\  r\left(t,\frac{\mu}{2}\right).
\end{align*} 
 
Consequently, there is a constant $C(n)$ such that
\begin{align*}
\mu^n|  \phi_\delta^{(n)}(\mu)|(1\ast\ell)(t)^\delta[\ell(t)]^{-1}&\le C(n)\  r\left(t,\frac{\mu}{2}\right)(1\ast\ell)(t)^\delta[\ell(t)]^{-1}\\
&\le C(n) \dfrac{\bigl(\mu(1\ast\ell)(t)\bigr)^\delta}{1+\mu(1*\ell)(t)}, \quad t>0,
\end{align*}
by \eqref{r:up:est}.  Since $\delta\in[0,1]$ we have the inequality

\[
\dfrac{\bigl(\mu(1\ast \ell)(t)\bigr)^\delta}{1+\mu(1\ast \ell)(t)}\le \dfrac{\bigl(\mu(1\ast \ell)(t)\bigr)^\delta}{\bigl(1+\mu(1\ast \ell)(t)\bigr)^\delta}<1,\quad \mu\ge0.
\]

In consequence, \eqref{phi:delta} is true, and the proof is complete. \end{proof}


\medskip

\subsection*{Fourier Multipliers} Let $m(\xi)$ be a complex-valued bounded function on $\mathbb{R}^d\smallsetminus \{0\}$. We say that $m(\xi)$ satisfies the Mihlin's condition if there exists $M=M(d)>0$ such that
\begin{equation}\label{Mihlin:multi}
|\xi|^{|\beta|} \big|\partial_\xi^\beta m(\xi)\big|\le M,\quad \xi\in \mathbb{R}^d\smallsetminus\{0\},\;\;|\beta|\le \left[\frac{d}{2}\right]+1,
\end{equation}
where $\beta=(\beta_1,\ldots,\beta_d)\in \mathbb{N}_0^d$ is a multiindex and $\partial_{\xi}^\beta m$ stands for the partial derivative of $m$ of order $|\beta|=\sum_{k=1}^d \beta_k$.

\begin{lemma} \label{Lem:aux5}
Let $(k,\ell)\in (\mathcal{PC})$, $\delta\in (0,1]$, $\rho >0$, and $t\ge 0$ be fixed. Let 
\[
m_0(\xi) :=\psi_\delta(|\xi|^\rho)=|\xi|^{\rho\delta} s(t,|\xi|^\rho),\quad \xi\in\mathbb{R}^d.
\]
Then $m_0\in C^\infty((0,\infty)^d)$ and the partial derivative $\partial_{\xi}^\beta m_0$ of order $|\beta|=\sum_{k=1}^N \beta_k$ is a sum of finitely many terms of the forms
\begin{equation} \label{summandform:1}
c_i(\rho,\beta)\cdot |\xi|^{\rho i}\psi_\delta^{(i)}\left(|\xi|^\rho\right)\cdot |\xi|^{-2j}\prod_{k=1}^d \xi_k^{\gamma_k}\quad
\mbox{with}\;\; i=0,\ldots, j, \; \mbox{and } \;\frac{|\beta|}{2}
\le j \le |\beta|,
\end{equation}
where $\sum_{k=1}^d \gamma_k=2j-|\beta|$ and the constants $c_i(\rho,\beta)$ could be but not all equal to zero. 

Moreover, the function $m(\xi):=m_0(\xi) [(1\ast \ell)(t)]^\delta$ is uniformly bounded w.r.t.\ $t\ge 0$ and satisfies the Mihlin's condition \eqref{Mihlin:multi}.
\end{lemma}

\begin{proof} The assertion on the structure of $\partial_{\xi}^\beta m_0$ can be proved by induction over $|\beta|$.
If $|\beta|=0$, then $\partial_{\xi}^\beta m_0(\xi)=m_0(\xi)=\psi_{\delta}(|\xi|^\rho)$, which is of the desired form with $j=0$.
Suppose now that the assertion is true for all $\beta\in\mathbb{N}_0^d$ of the same fixed order $b:=|\beta|\in \mathbb{N}_0$. Let
$\beta'\in \mathbb{N}_0^d$ with $|\beta'|=b+1$. Then $\partial_{\xi}^{\beta'} m_0=\partial_{\xi_l}\partial_{\xi}^\beta m_0$
for some $\beta\in \mathbb{N}_0^d$ with $|\beta|=b$ and some $l\in \{1,\ldots,d\}$. By the induction hypothesis, $\partial_{\xi_l}\partial_{\xi}^\beta m_0$ is a finite sum of first order partial derivatives w.r.t.\ $\xi_l$ of terms of the
form described in (\ref{summandform:1}). Let us consider such a term. Let $\gamma_l\geq 0$, we have
\begin{align}
\partial_{\xi_l} &\left[|\xi|^{\rho i}\psi_\delta^{(i)}\left(|\xi|^\rho\right) \cdot |\xi|^{-2j}\prod_{k=1}^d \xi_k^{\gamma_k}\right]  =  \rho\left( |\xi|^{\rho (i+1)} \psi_\delta^{(i+1)}\left(|\xi|^{\rho}\right)\cdot |\xi|^{-2(j+1)}  \xi_l \cdot \prod_{k=1}^d \xi_k^{\gamma_k}\right)\nonumber\\
&  + (\rho i - 2j) \left(|\xi|^{\rho i} \psi_\delta^{(i)}\left(|\xi|^{\rho}\right)\cdot |\xi|^{-2(j+1)}  \xi_l \cdot \prod_{k=1}^d \xi_k^{\gamma_k}\right)\nonumber\\
& + \gamma_l\left(|\xi|^{\rho i} \psi_\delta^{(i)}\left(|\xi|^{\rho}\right)\cdot |\xi|^{-2j}  \xi_l^{\gamma_l -1} \cdot \prod_{k=1, k\neq l}^d \xi_k^{\gamma_k}\right),\quad i=0,\ldots, j.\label{partial:1}
\end{align}
We see that the constants on the right-hand side of \eqref{partial:1} could be zero if for instance $\rho = 2$ and $i=j$, or $\gamma_l=0$.

The first and the second term on the right-hand side of (\ref{partial:1}) has the desired form, since with $\gamma_k':=\gamma_k$, $
k\neq l$ and $\gamma_l'=\gamma_l+1$, we have by the induction hypothesis
\[
0\le \sum_{k=1}^d \gamma'_k=\sum_{k=1}^d \gamma_k+1=2j-|\beta|+1=2(j+1)-|\beta'|.
\]
The third term has the desired form as well, since setting $\gamma_k'=\gamma_k$, $
k\neq l$ and $\gamma_l'=\gamma_l-1$, we have now
\[
0\le \sum_{k=1}^d \gamma'_k=\sum_{k=1}^d \gamma_k-1=2j-|\beta|-1=2j-|\beta'|.
\]
The second part of the proof follows from the first one and Lemma \ref{Lem:aux2}. In fact, for any term $T(\xi)$ of the form (\ref{summandform:1}), Lemma \ref{Lem:aux2} yields the estimate
\begin{align*}
|\xi|^{|\beta|}|T(\xi)| & \le |c_i(\rho,\beta)|\,|\xi|^{|\beta|}\,|\xi|^{\rho i}|\psi_\delta^{(i)}|\left(|\xi|^\rho\right)\, |\xi|^{-2j}\prod_{k=1}^d |\xi_k|^{\gamma_k}\\
& \le c(\rho, \beta)\,|\xi|^{\rho i}|\psi_\delta^{(i)}\left(|\xi|^\rho\right)|\cdot |\xi|^{-2j + |\beta|}\,|\xi|^{\sum_{k=1}^d \gamma_k}\\
& \le  \,\frac{C(\rho,\beta,j)}{[(1\ast \ell)(t)]^\delta}.
\end{align*}
It is now evident that $m(\xi)$ satisfies Mihlin's condition \eqref{Mihlin:multi} with a constant $M$ that merely depends on the dimension $d$. 
\end{proof}


In case $\rho = 2$ and $m(\xi)=\psi_\delta(|\xi|^2)[(1\ast \ell)(t)]^{\delta}$ the function $m$ satisfies condition \eqref{Mihlin:multi}, and further the partial derivative of $m_0(\xi):= \psi_{\delta}(|\xi|^2)$ is a sum of finitely many terms of the forms
\[
c(\beta)\cdot \psi_\delta^{(j)}\left(|\xi|^2\right)\cdot \prod_{i=1}^d \xi_i^{\gamma_i}\quad
\mbox{with}\;\;\frac{|\beta|}{2}
\le j \le |\beta|\;\;\mbox{and}\;\;\sum_{i=1}^d \gamma_i=2j-|\beta|,
\]
where $c(\beta)>0$. See \cite[Lemma 5.3]{Kem-Sil-Ver-Zach-2016}. Our lemma has extended this result for all $\rho >0$.

\medbreak

\noindent Now we consider $m(\xi)$ as the function defined $m(\xi) :=|\xi|^{\rho\delta} r(t,|\xi|^\rho)[(1\ast \ell)(t)]^{\delta}[\ell(t)]^{-1}$ with $\xi\in \mathbb{R}^d$. Our purpose is to establish that this function is uniformly bounded w.r.t.\ $t\ge 0$, and satisfies the Mihlin's condition \eqref{Mihlin:multi}. The proof is very similar to the proof of Lemma \ref{Lem:aux5} and we omit it for the sake of the brevity of the text.


\begin{lemma}\label{Lem:aux6}
Let $(k,\ell)\in (\mathcal{PC})$, $\delta\in [0,1]$, $\rho >0$, and $t\ge 0$ be fixed. Let 
\[
m_{0}(\xi) :=\phi_\delta(|\xi|^\rho)=|\xi|^{\rho\delta} r(t,|\xi|^\rho),\quad \xi\in \mathbb{R}^d.
\]
Then $m_0\in C^\infty((0,\infty)^d)$ and the partial derivative $\partial_{\xi}^\beta m_0$ of order $|\beta|=\sum_{k=1}^N \beta_k$ is a sum of finitely many terms of the forms
\begin{equation} \label{summandform:2}
c_i(\rho,\beta)\cdot |\xi|^{\rho i}\phi_\delta^{(i)}\left(|\xi|^\rho\right)\cdot |\xi|^{-2j}\prod_{k=1}^d \xi_k^{\gamma_k}\quad
\mbox{with}\;\; i=0,\ldots, j, \; \mbox{and } \;\frac{|\beta|}{2}
\le j \le |\beta|,
\end{equation}
where $\sum_{k=1}^d \gamma_k=2j-|\beta|$ and the constants $c_i(\rho,\beta)$ could be but not all equal to zero. 

Moreover, for all $\delta\in[0,1]$, the function $m(\xi):=\phi_\delta(|\xi|^\rho)[(1\ast\ell)(t)]^{\delta}[\ell(t)]^{-1}$ is uniformly bounded w.r.t.\ $t\ge 0$ and satisfies
the Mihlin's condition (\ref{Mihlin:multi}).

\end{lemma} 


Now we are in conditions to prove our main results. 

\begin{theorem}\label{Theo:Lr:Est:u0}
Let $(k, \ell)\in (\mathcal{PC})$ and $u(t,x)=\big(Z(t,\cdot)\star u_0\bigr)(x)$, where $u_0$ is described below.
\begin{itemize}
\item[(i)] Let $d\in \mathbb{N}$, $1< p<\sigma_1(\rho,d)$, $1<q, r<\infty$, such that $1+\frac{1}{r}=\frac{1}{p}+
\frac{1}{q}$, and $u_0\in L_1(\mathbb{R}^d)\cap L_q(\mathbb{R}^d)$. Then
\[
|u(t,\cdot)|_r \lesssim \big[(1\ast \ell)(t)\big]^{ -\frac{ d }{\rho}\,\left(1-\frac{1}{p}\right)},\quad t>0.
\]
\item[(ii)] Let $d>\rho$, $1< q,r < \infty$, such that $\frac{1}{r}+\frac{\rho}{d}=\frac{1}{q}$, and $u_0\in L_1(\mathbb{R}^d)\cap L_q(\mathbb{R}^d)$. Then
\[
|u(t,\cdot)|_r \lesssim \big[(1\ast \ell)(t)\big]^{ -1},\quad t>0.
\]
\item[(iii)] Let $d> \rho$ and $u_0\in L_1(\mathbb{R}^d)$. Then
\[
|u(t,\cdot)|_{\frac{d}{d-\rho},\infty}\lesssim \big[(1\ast \ell)(t)\big]^{ -1} ,\quad t>0.
\]

\end{itemize}

\end{theorem}

\begin{proof} (i) Set $\delta=\frac{ d }{\rho}(1-\frac{1}{p})$. If $d\le \rho$ it is clear that $\delta\in (0,1)$ for all $p\in (1,\infty)$. If $d>\rho$ we have by assumption $1<p<\sigma_1(\rho,d)$, which is equivalent to $\delta\in (0,1)$. With $t>0$ being fixed we write
\begin{equation} \label{decomp2}
\widetilde{u}(t,\xi)=[(1\ast \ell)(t)]^{-\delta}\,\big(\psi_\delta(|\xi|^\rho) [(1\ast \ell)(t)]^\delta\big)\,\big(|\xi|^{-\rho\delta}\widetilde{u_0}(\xi)\big).
\end{equation}
By the Hardy-Littlewood-Sobolev theorem on fractional integration, see e.g.\ \cite[Theorem.\ 6.1.3]{Gra-2004},
$(-\Delta)^{-\frac{\rho\delta}{2}}u_0\in L_r(\mathbb{R}^d)$ and $|(-\Delta)^{-\frac{\rho\delta}{2}}u_0|_r\le C(d,\delta,q)|u_0|_q$; in fact, the choice of $\delta$ and the assumption  $1+\frac{1}{r}=\frac{1}{p}+\frac{1}{q}$ imply that
\[
\frac{1}{q}\,-\,\frac{1}{r}\,=\,\frac{\rho\delta}{d}\quad \mbox{and}\;\;\rho\delta<d.
\]
Thanks to Lemma \ref{Lem:aux5} we know that $m(\xi)=\psi_\delta(|\xi|^\rho) [(1\ast \ell)(t)]^\delta$
satisfies Mihlin's condition with a dimensional constant that is independent of $t>0$. Thus we may apply
Mihlin's multiplier theorem, see \cite[Theorem 5.2.7]{Gra-2004}, thereby obtaining that
\[
|u(t,\cdot)|_r\le C(d,r)[(1\ast \ell)(t)]^{-\delta}|(-\Delta)^{-\frac{\rho\delta}{2}}u_0|_r \lesssim [(1\ast \ell)(t)]^{-\delta}.
\]
This proves (i).

(ii) We consider again the decomposition (\ref{decomp2}), now setting $\delta=1$. As before we see that the Hardy-Littlewood-Sobolev inequality implies $(-\Delta)^{-\frac{\rho}{2}}u_0\in L_r(\mathbb{R}^d)$. The assertion follows then from
Lemma \ref{Lem:aux5} with $\delta=1$ and Mihlin's multiplier theorem.

(iii) We know already that $m(\xi)=\psi_1(|\xi|^\rho) [(1\ast \ell)(t)]$ is an $L_r(\mathbb{R}^d)$-Fourier multiplier for all $r\in (1,\infty)$ with a constant that only depends on $r$ and $d$, that is, the operator $T$ defined by $Tf=\mathcal{F}^{-1}(m \mathcal{F}f)$
($\mathcal{F}$ denoting the Fourier transform) on a suitable dense subset of $L_r(\mathbb{R}^d)$ is $L_r(\mathbb{R}^d)$-bounded, thus extends to an operator $T\in\mathcal{B}(L_r(\mathbb{R}^d))$, and $|T|_{\mathcal{B}(L_r)}\le M(d,r)$. The weak $L_r$-spaces
can be obtained from the strong ones by real interpolation. Assuming $1<r<\infty$ we may choose $r_1\in (1,r)$,
$r_2\in (r,\infty)$, and $\theta\in (0,1)$ such that $\frac{1}{r}=\frac{1-\theta}{r_1}+\frac{\theta}{r_2}$. By \cite[Theorem 1.18.2]{Tri-1995}, we then have $(L_{r_1},L_{r_2})_{\theta,\infty}=L_{r,\infty}$. It follows that $T\in \mathcal{B}(L_{r,\infty}(\mathbb{R}^d))$, with a norm bound that only depends on $r$ and $d$.

We choose $r=\frac{d}{d-\rho}$. Then $1-\frac{1}{r}=\frac{\rho}{d}$, and
the Hardy-Littlewood-Sobolev theorem (\citep[Theorem\ 6.1.3]{Gra-2004}) implies that $(-\Delta)^{-1}u_0\in L_{r,\infty}(\mathbb{R}^d)$.
Note that $u_0\in L_1(\mathbb{R}^d)$ and so we only get an estimate in a weak $L_r$-space.
The assertion now follows from (\ref{decomp2}) with $\delta=1$ and the fact that $T\in \mathcal{B}(L_{r,\infty}(\mathbb{R}^d))$, with a norm bound that is independent of $t>0$. 
\end{proof}

\begin{theorem}\label{Theo:Lr:Est:f}
Let $(k,\ell)\in(\mathcal{PC})$, $d\in\mathbb{N}$ and $\rho>0$. Let $u$ be the solution of \eqref{Equation:1}--\eqref{Equation:2} with $u_0\equiv0$ and $f(t,\cdot)\in L_1(\mathbb{R}^d)\cap L_q(\mathbb{R}^d)$ for all $t\ge 0$ and some $q\in(1,\infty)$. Then 
\begin{equation}\label{Est:Lq:Lq:f}
|u(t,\cdot)|_q\lesssim \int_0^t \ell(s)|f(t-s,\cdot)|_q ds.
\end{equation}
Additionally, if $1\le p\le\sigma_1(\rho,d)$, $1<q, r<\infty$, such that $1+\frac{1}{r}=\frac{1}{p}+
\frac{1}{q}$ then 
\begin{equation}\label{Est:Lq:Lr:f}
|u(t,\cdot)|_r\lesssim \int_0^t \ell(s)[(1\ast \ell)(s)]^{-\delta}|f(t-s,\cdot)|_q ds,
\end{equation}
where $\delta=\frac{d}{\rho}\left(1-\frac{1}{p}\right)$.
\end{theorem}


\begin{proof}
With $t>0$ being fixed we note that 
\begin{align*}
|u(t,\cdot)|_q&=\left|\int^t_0 \mathcal{F}^{-1}\bigl(\widetilde{Y}(s,\xi)\widetilde{f}(t-s,\xi)\bigr)(\cdot)ds\right|_q\\
&=\left|\int^t_0 \ell(s)\mathcal{F}^{-1}\bigl([\ell(s)]^{-1}r(s,|\xi|^\rho)\widetilde{f}(t-s,\xi)\bigr)(\cdot)ds\right|_q.
\end{align*}
Using Minkowsky's inequality for integrals (see \cite[Theorem 6.19]{Foll-1999}), we have that 
\[
|u(t,\cdot)|_q\le\int^t_0 \ell(s)\left|\mathcal{F}^{-1}\bigl([\ell(s)]^{-1}r(s,|\xi|^\rho)\widetilde{f}(t-s,\xi)\bigr)(\cdot)\right|_q ds.
\]
It follows from Lemma \ref{Lem:aux6} that $m(\xi)=r(t,|\xi|^\rho)[\ell(t)]^{-1}$ is an $L_q$-Fourier multiplier for all $q\in(1,\infty)$. Therefore, there exists a constant $M(q,d)$ that depends on $q$ and $d$ such that 
\[
|u(t,\cdot)|_q\le M(q,d)\int^t_0 \ell(s)\left|f(t-s,\cdot)\right|_q ds\lesssim \int^t_0 \ell(s)\left|f(t-s,\cdot)\right|_q .
\]

On the other hand, if $1\le p\le\sigma_{1}(d,\rho)$ and $\delta=\frac{d}{\rho}\left(1-\frac{1}{p}\right)$, then $\delta\in [0,1]$. Further,
\begin{align*}
|u(t,\cdot)|_r=&\left|\int^t_0 \ell(s)[(1\ast\ell)(s)]^{-\delta}\mathcal{F}^{-1}\bigl([\ell(s)]^{-1}[(1\ast\ell)(s)]^{\delta}\phi_\delta(|\xi|^\rho)|\xi|^{-\rho\delta}\widetilde{f}(t-s,\xi)\bigr)(\cdot)ds\right|_r.
\end{align*}

Using again Minkowsky's inequality for integrals, we have that 
\[
|u(t,\cdot)|_r\le\int^t_0 \ell(s)[(1\ast \ell)(s)]^{-\delta}\left|\mathcal{F}^{-1}\bigl([\ell(s)]^{-1}[(1\ast \ell)(s)]^{\delta}\phi_\delta(|\xi|^\rho)|\xi|^{-\rho\delta}\widetilde{f}(s,\xi)\bigr)(\cdot)\right|_r ds.
\]
Since $\delta\in[0,1]$, it follows from Lemma \ref{Lem:aux6} that $m(\xi)=\phi(|\xi|^\rho)[\ell(t)]^{-1}[(1\ast\ell)(t)]^{\delta}$ is an $L_r$-Fourier multiplier for all $r\in(1,\infty)$. Therefore, there exists a positive constant $M(q,d)$ such that 
\[
|u(t,\cdot)|_r\le M(q,d)\int^t_0 \ell(s)[(1\ast \ell)(s)]^{-\delta}\left|(-\Delta)^{-\frac{\rho\delta}{2}}f(t-s,\cdot)\right|_r ds.
\]
Since $1+\frac{1}{r}=\frac{1}{p}+
\frac{1}{q}$,  we can apply Hardy-Littlewood-Sobolev's Inequality for fractional integration to prove that there is a positive constant $C(d,\delta,r)$ such that
\[
|u(t,\cdot)|_r\le M(q,d)C(d,\delta,r)\int^t_0 \ell(s)[(1\ast \ell)(s)]^{-\delta}\left|f(t-s,\cdot)\right|_q ds.
\]
Therefore 
\[
|u(t,\cdot)|_r\lesssim \displaystyle\int^t_0 \ell(s)[(1\ast \ell)(s)]^{-\delta}\left|f(t-s,\cdot)\right|_q ds.
\]
\end{proof}

\begin{remark}
We point out that the estimate \eqref{Est:Lq:Lr:f} decays to zero faster than estimate \eqref{Est:Lq:Lq:f} if only if $\ell\notin L_1(\mathbb{R}_+)$. 
\end{remark}

\medskip


\subsection*{$L_r$- estimates of gradient of the solution} Next we obtain $L_r$-estimates of gradient of solutions to \eqref{Equation:1}-\eqref{Equation:2}. With this end, we denote

\begin{equation}\label{sigma:2}
\sigma_2(\rho,d)=\begin{cases}\dfrac{d}{d-\rho+1},& d>\rho-1,\\ \infty, & \text{otherwise}.\end{cases}
\end{equation}

\begin{theorem}\label{Theo:Grad:Sol}
Let $d\in \mathbb{N}$, and $(k, \ell)\in (\mathcal{PC})$. Let $u(t,x)=\big(Z(t,\cdot)\star u_0\bigr)(x)$. Assume that $u_0\in L_1(\mathbb{R}^d)\cap L_q(\mathbb{R}^d)$. The following statements hold.

\begin{itemize}
\item[(i)] If $\rho\ge 1$, then 
\[
|\nabla u(t,\cdot)|_q \lesssim \big[(1\ast \ell)(t)\big]^{ -\frac{1}{\rho}},\quad t>0.
\]

\item[(ii)] Let $\rho\ge1$, $1\le p<\sigma_2(\rho,d)$ and $1<q, r<\infty$, such that $1+\frac{1}{r}=\frac{1}{p}+
\frac{1}{q}$. Then
\[
|\nabla u(t,\cdot)|_r \lesssim \big[(1\ast \ell)(t)\big]^{-\frac{1}{\rho} -\frac{ d }{\rho}\,\left(1-\frac{1}{p}\right)},\quad t>0.
\]
\item[(iii)] Let $d>\rho-1$ and $1< q,r < \infty$ such that $\frac{1}{r}+\frac{\rho-1}{d}=\frac{1}{q}$. Then
\[
|\nabla u(t,\cdot)|_r \lesssim \big[(1\ast \ell)(t)\big]^{ -1},\quad t>0.
\]
\item[(iv)] Let $d> \rho-1$ and $u_0\in L_1(\mathbb{R}^d)$. Then
\[
|\nabla u(t,\cdot)|_{\frac{d}{d-\rho+1},\infty}\lesssim \big[(1\ast \ell)(t)\big]^{ -1} ,\quad t>0.
\]

\end{itemize}

\end{theorem}


\begin{proof} We note that there is a positive constant $C(q,d)$ such that
\[
|\nabla u(t,\cdot)|_q\le C(q,d)\left(\int_{\mathbb{R}^d} \left(\max_{1\le j\le d} |\partial_{x_j} u(t,x)|\right)^{q} dx\right)^{1/q}.
\]
Let $j\in\{1,2,\cdots,d\}$ and $t>0$ be fixed. In order to prove the statement $(i)$, we note that
\[
\partial_{x_j} u(t,x)=\mathcal{F}^{-1} \Bigl(-i\xi_j s(t,|\xi|^{\rho})\Bigr)=[(1\ast\ell)(t)]^{-\frac{1}{\rho}}\mathcal{F}^{-1} \Bigl(-i\xi_j\ [(1\ast\ell)(t)]^{\frac{1}{\rho}}\ s(t,|\xi|^{\rho})\widetilde{u_0}(\xi)\Bigr).
\]
Define $m_j(\xi)=-i\xi_j\ [(1\ast\ell)(t)]^{\frac{1}{\rho}}\ s(t,|\xi|^{\rho})$. We note that 
\[
m_{j}(\xi)=-i\xi_j|\xi|^{-1}\ (|\xi|^{\rho})^{\frac{1}{\rho}}\ s(t,|\xi|^\rho)\ [(1\ast\ell)(t)]^{\frac{1}{\rho}}.
\]

Since $\rho\ge 1$, it follows from Lemma \ref{Lem:aux5} that 
\[
m_0(\xi):=(|\xi|^{\rho})^{\frac{1}{\rho}}s(t,|\xi|^\rho)\ [(1\ast\ell)(t)]^{\frac{1}{\rho}},
\]
is uniformly bounded w.r.t.\ $t\ge 0$ and it satisfies the Mihlin's condition. Further, by a direct a computation, the function $\xi\mapsto\xi_j|\xi|^{-1}$ satisfies the Mihlin's condition. Using the Leibniz's formula for differentiation of product of two functions we conclude that $m_j$ satisfies the Mihlin's condition. Thus we may apply
Mihlin's multiplier theorem to obtain 
\[
|\partial_{x_j} u(t,\cdot)|_q\le C(q,d)[(1\ast\ell)(t)]^{-\frac{1}{\rho}}|u_0|_q,
\]
where $C(q,d)$ is a constant depending only $q$ and $d$. Since the last estimative is independent of $j$, in particular it is valid for $\displaystyle\max_{1\le j\le d}|\partial_j u(t,x)|$. Hence, we have 
\[
|\nabla u(t,\cdot)|_q\lesssim [(1\ast\ell)(t)]^{-\frac{1}{\rho}}.
\]

To prove (ii) we define $\delta=\frac{ d }{\rho}\left(1-\frac{1}{p}\right)$. Since $\rho\ge 1$  and $1\le p<\sigma_2(d,\rho)$, we have that $\frac{1}{\rho}+\delta\in (0,1]$. With $t>0$ being fixed, we write
\begin{equation*} \label{decomp3}
\widetilde{\partial_{x_j}u}(t,\xi)=[(1\ast\ell)(t)]^{-(\frac{1}{\rho}+\delta)}\Bigl(-i\xi_j|\xi|^{-1}\ [(1\ast\ell)(t)]^{\frac{1}{\rho}+\delta} (|\xi|^\rho)^{\frac{1}{\rho}+\delta}\ s(t,|\xi|^{\rho})\ |\xi|^{-\delta\rho}\widetilde{u_0}(\xi)\Bigr).
\end{equation*}

Using Hardy-Littlewood-Sobolev inequality on fractional integrals, we have that  
$(-\Delta)^{-\frac{\rho\delta}{2}}u_0\in L_r(\mathbb{R}^d)$ and $|(-\Delta)^{-\frac{\rho\delta}{2}}u_0|_r\le C(d,\delta,q)|u_0|_q$.

Since $\frac{1}{\rho}+\delta\in (0,1]$, thanks to Lemma \ref{Lem:aux5} we know that $[(1\ast\ell)(t)]^{\frac{1}{\rho}+\delta} (|\xi|^\rho)^{\frac{1}{\rho}+\delta}\ s(t,|\xi|^{\rho})$
satisfies Mihlin's condition with a dimensional constant that is independent of $t>0$. Thus we may apply
Mihlin's multiplier theorem, thereby obtaining that
\[
|\partial_{x_j}u(t,\cdot)|_r\le C(d,r)[(1\ast \ell)(t)]^{-\frac{1}{\rho}-\delta}|(-\Delta)^{-\frac{\rho\delta}{2}}u_0|_r \lesssim [(1\ast \ell)(t)]^{-\frac{1}{\rho}-\delta}.
\]
This proves (ii). 

To prove (iii) we note that the hypotheses imply that $\frac{1}{\rho}+\delta=1$. As before we see that the Hardy-Littlewood-Sobolev inequality implies $(-\Delta)^{-\frac{d}{2}\left(1-\frac{1}{p}\right)}u_0\in L_r(\mathbb{R}^d)$. From here, the proof follows the same lines of the proof of the statement (ii).

(iv) We know already that $m_j(\xi)=-i\xi_j |\xi|^{-1}(|\xi|^{\rho})^{1/\rho}\ s(t,|\xi|^\rho) [(1\ast \ell)(t)]$ is an $L_r(\mathbb{R}^d)$-Fourier multiplier for all $r\in (1,\infty)$ with a constant that only depends on $r$ and $d$., that is, the operator $T$ defined by $Tf=\mathcal{F}^{-1}(m_j \mathcal{F}f)$
($\mathcal{F}$ denoting the Fourier transform) on a suitable dense subset of $L_r(\mathbb{R}^d)$ is $L_r(\mathbb{R}^d)$-bounded, thus extends to an operator $T\in\mathcal{B}(L_r(\mathbb{R}^d))$, and $|T|_{\mathcal{B}(L_r)}\le M(d,r)$. The weak $L_r$-spaces
can be obtained from the strong ones by real interpolation. Assuming $1<r<\infty$ we may choose $r_1\in (1,r)$,
$r_2\in (r,\infty)$, and $\theta\in (0,1)$ such that $\frac{1}{r}=\frac{1-\theta}{r_1}+\frac{\theta}{r_2}$. By \cite[Theorem 1.18.2]{Tri-1995}, we then have $(L_{r_1},L_{r_2})_{\theta,\infty}=L_{r,\infty}$. It follows that $T\in \mathcal{B}(L_{r,\infty}(\mathbb{R}^d))$, with a norm bound that only depends on $r$ and $d$.

We choose $r=\frac{d}{d-\rho+1}$. Then $1-\frac{1}{r}=\frac{\rho-1}{d}$, and
the Hardy-Littlewood-Sobolev theorem (\citep[Theorem\ 6.1.3]{Gra-2004}) implies that $(-\Delta)^{-\frac{d}{2}\left(1-\frac{1}{p}\right)}u_0\in L_{r,\infty}(\mathbb{R}^d)$.
Note that $u_0\in L_1(\mathbb{R}^d)$ and so we only get an estimate in a weak $L_r$-space.
The assertion now follows from (\ref{decomp3}) with $\frac{1}{\rho}+\delta=1$ and the fact that $T\in \mathcal{B}(L_{r,\infty}(\mathbb{R}^d))$, with a norm bound that is independent of $t>0$. 
\end{proof}


\begin{theorem}\label{Theo:Lr:Grad:Est:f}
Let $(k,\ell)\in(\mathcal{PC})$, $d\in\mathbb{N}$ and $\rho\ge1$. Assume that $u$ is the solution of \eqref{Equation:1}--\eqref{Equation:2} with $u_0\equiv0$ given by (\ref{Variation}). The following assertions hold.
\begin{enumerate}[(i)]
\item If $f(t,\cdot)\in L_1(\mathbb{R}^d)\cap L_q(\mathbb{R}^d)$ for all $t\ge 0$ and some $q\in(1,\infty)$,  then 
\[
|\nabla u(t,\cdot)|_q\lesssim \int_0^t \ell(s)[(1\ast \ell)(s)]^{-\frac{1}{\rho}}\ |f(t-s,\cdot)|_q ds.
\]
\item If $f(t,\cdot)\in L_1(\mathbb{R}^d)\cap L_q(\mathbb{R}^d)$, $1\le p\le\sigma_2(\rho,d)$ and $1<q, r<\infty$, such that $1+\frac{1}{r}=\frac{1}{p}+
\frac{1}{q}$, then 
\[
|\nabla u(t,\cdot)|_r\lesssim \int_0^t \ell(s)[(1\ast \ell)(s)]^{-\frac{1}{\rho}-\delta}\ |f(t-s,\cdot)|_q ds,
\]
where $\delta=\frac{d}{\rho}\left(1-\frac{1}{p}\right)$.
\end{enumerate}
\end{theorem}

\begin{proof}
As in Theorem \ref{Theo:Grad:Sol}, we note that there is a positive constant $C(q,d)$ such that
\[
|\nabla u(t,\cdot)|_q\le C(q,d)\left(\int_{\mathbb{R}^d} \left(\max_{1\le j\le d} |\partial_{x_j} u(t,x)|\right)^{q} dx\right)^{1/q}.
\]

Let $j\in\{1,2,\cdots,d\}$ and $t>0$ be fixed. To prove the statement (i) we note that
\begin{align*}
\partial_{x_j} u(t,x)&=\int_0^t \mathcal{F}^{-1} \Bigl(-i\xi_j r(s,|\xi|^{\rho})\widetilde{f}(t-s,\xi)ds\Bigr)\\
&=\int_0^t \ell(s)[(1\ast\ell)(s)]^{-\frac{1}{\rho}}\ \mathcal{F}^{-1} \Bigl(-i\xi_j\ r(s,|\xi|^{\rho}) [\ell(s)]^{-1}[(1\ast\ell)(s)]^{\frac{1}{\rho}}\ \widetilde{f}(t-s,\xi)ds\Bigr).
\end{align*}
Define $m_j(t,\xi)=-i\xi_j\ r(t,|\xi|^{\rho}) [\ell(t)]^{-1}[(1\ast\ell)(t)]^{\frac{1}{\rho}}$. We note that 
\[
m_{j}(t,\xi)=-i\xi_j|\xi|^{-1}\ (|\xi|^{\rho})^{\frac{1}{\rho}}\ r(t,|\xi|^\rho)\ [\ell(t)]^{-1}[(1\ast\ell)(t)]^{\frac{1}{\rho}}.
\]

Since $\rho\ge 1$, it follows from Lemma \ref{Lem:aux5} that the function $m_j(t,\xi)$ satisfies the Mihlin's condition \ref{Mihlin:multi}. Using Minkowsky's inequality for integrals we have that 
\[
|\partial_{x_j} u(t,\cdot)|_q\le \int_0^t \ell(s)[(1\ast\ell)(s)]^{-\frac{1}{\rho}}\left|\mathcal{F}^{-1} \Bigl(m_j(\cdot)\widetilde{f}(t-s,\cdot)\Bigr)\right|_q.
\]

By Mihlin's multiplier theorem we have that $m_j(t,\xi)$ is an $L_q$-Fourier multiplier for all $q\in(1,\infty)$. Thus there is a constant $M$ such that 
\[
|\partial_{x_j} u(t,x)|_q\le M\int_0^t \ell(s)[(1\ast\ell)(s)]^{-\frac{1}{\rho}}\left|f(t-s,\cdot)\right|_q.
\]
We point out that the last estimative is independent of $j$. Hence, in particular it is valid for $\displaystyle\max_{1\le j \le d}|\partial_{x_j}u(t,x)|$. Therefore, we have 
\[
|\nabla u(t,\cdot)|_q\lesssim \int_0^t \ell(s)[(1\ast\ell)(s)]^{-\frac{1}{\rho}}\left|f(t-s,\cdot)\right|_q.
\]

To prove (ii) we define $\delta=\frac{ d }{\rho}(1-\frac{1}{p})$. Since $\rho\ge 1$  and $1<p<\sigma_2(d,\rho)$, we have that $\frac{1}{\rho}+\delta\in [0,1]$. We write
\[
\partial_{x_j}u(t,\xi)=-i\int_0^t \ell(s)[(1\ast\ell)(s)]^{-\left(\frac{1}{\rho}+\delta\right)}\mathcal{F}^{-1}\Bigl(\eta_j(s,\xi)\ |\xi|^{-\delta\rho}\widetilde{f}(t-s,\xi)ds\Bigr),
\]
where $\eta_j(t,\xi)=\xi_j|\xi|^{-1}\ [\ell(t)]^{-1}[(1\ast\ell)(t)]^{\frac{1}{\rho}+\delta} (|\xi|^\rho)^{\frac{1}{\rho}+\delta}\ r(t,|\xi|^{\rho})$. Using Minkowsky's inequality for integrals we have
\[
|\partial_{x_j}u(t,\xi)|_r\le \int_0^t \ell(s)[(1\ast\ell)(s)]^{-(\frac{1}{\rho}+\delta)}\left|\mathcal{F}^{-1}\Bigl(\eta_j(s,\xi)\ |\xi|^{-\delta\rho}\widetilde{f}(t-s,\xi)\Bigr)\right|_r ds.
\]
Since $\frac{1}{\rho}+\delta\in [0,1]$, thanks to Lemma \ref{Lem:aux5} we know that $\ell(t)[(1\ast\ell)(t)]^{\frac{1}{\rho}+\delta} (|\xi|^\rho)^{\frac{1}{\rho}+\delta}\ r(t,|\xi|^{\rho})$
satisfies Mihlin's condition with a dimensional constant that is independent of $t>0$. Thus we may apply
Mihlin's multiplier theorem, thereby obtaining that
\[
|\partial_{x_j}u(t,\xi)|_r\le \int_0^t \ell(s)[(1\ast\ell)(s)]^{-(\frac{1}{\rho}+\delta)}\left| (-\Delta)^{-\delta\rho}\widetilde{f}(t-s,\xi)\right|_r ds.
\]

By Hardy-Littlewood-Sobolev inequality on fractional integrals, we have that  
\[
(-\Delta)^{-\frac{\rho\delta}{2}}f(t,\cdot)\in L_r(\mathbb{R}^d),  \text{\ and\ } 
|(-\Delta)^{-\frac{\rho\delta}{2}}f(t,\cdot)|_r\le C(d,\delta,q)|f(t,\cdot)|_q.
\]
Therefore,
\[
|\partial_{x_j}u(t,\xi)|_r\lesssim \int_0^t \ell(s)[(1\ast\ell)(s)]^{-\left(\frac{1}{\rho}+\delta\right)}\left|f(t-s,\cdot)\right|_q ds.
\]

Since this estimate is independent of $j$, in particular it is valid for $\displaystyle \max_{1\le j\le d}|\partial_{x_j}u(t,x)|$. Therefore, we have
\[
|\nabla u(t,\xi)|_r\lesssim \int_0^t \ell(s)[(1\ast\ell)(s)]^{-(\frac{1}{\rho}+\delta)}\left|f(t-s,\cdot)\right|_q ds.
\]

\end{proof}


\section{Examples}\label{S:Ex}

In this section we discuss in detail the asymptotic behavior of the solutions to \eqref{Equation:1}-\eqref{Equation:2}, for several examples of pairs of kernels $(k,\ell)\in(\mathcal{PC})$. We will see that there are kernels that allow very different kinds of decay, e.g., exponential, algebraic, and logarithmic decay.

\bigskip

According to Theorem \ref{Theo:Lr:Est:u0} and Theorem \ref{Theo:Grad:Sol} the decay rate of the solutions and the gradient of the solutions to \eqref{Equation:1}--\eqref{Equation:2} with $f\equiv 0$, are determined by the behavior of powers of the function $[(1\ast\ell)(t)]^{-1}$, as $t\to\infty$. In the Examples \ref{Ex:1}, \ref{Ex:2} , \ref{Ex:3} and \ref{Ex:4} we have analized how is the behavior of these functions for several kernels $\ell$. For this reason, we will focus our attention into the no-homogeneous problem, that is the problem \eqref{Equation:1}-\eqref{Equation:2} with $u_0\equiv 0$.


\begin{corollary}\label{Theo:Ex:decay} Let $1\le p\le \sigma_1(d,\rho)$ and $1<q,r<\infty$ be such that $1+\frac{1}{p}=\frac{1}{r}+\frac{1}{q}$. Assume further that $|f(t,\cdot)|_q\lesssim g_{\gamma}(t)$ for some $\gamma\in(0,1)$. The following assertions hold.
\begin{enumerate}[$(i)$]
\item If $k=g_{1-\alpha}$, then 
\begin{equation}\label{Ex:decay:1}
|u(t,\cdot)|_r\lesssim t^{\gamma-1+\alpha-\frac{\alpha d}{\rho}\left(1-\frac{1}{p}\right)},\text{\ as \ }t\to\infty.
\end{equation}
\item If $k=g_{1-\alpha}+g_{1-\beta}$, with $0<\alpha<\beta<1$, then 
\begin{equation}\label{Ex:decay:2}
|u(t,\cdot)|_r\lesssim t^{\gamma-1+\alpha-\frac{\alpha d}{\rho}\left(1-\frac{1}{p}\right)},\text{\ as \ }t\to\infty.
\end{equation}

\item If $k(t)=g_{\beta}(t)E_{\alpha,\beta}(-\omega t^{\alpha})$, for $t>0$,
with $0<\alpha<\beta<1$ and $\omega>0$, then 
\begin{equation}\label{Ex:decay:3}
|u(t,\cdot)|_r\lesssim t^{\gamma-1+\bigl(\beta-\alpha-1\bigr)\bigl(1-\frac{d}{\rho}\bigl(1-\frac{1}{p}\bigr)\bigr)}, \text{\ as \ }t\to\infty.
\end{equation}
\end{enumerate}
\end{corollary}

\begin{proof}
We begin proving $(i)$. In this case, a direct calculation shows that 
\[
\ell(t)=g_{\alpha}(t)\sim t^{1-\alpha}, \text{ \ as \ } t\to\infty.,
\]
and
\[
(1\ast\ell)(t)=g_{1+\alpha}(t)\sim t^{\alpha}, \text{ \ as \ } t\to\infty.
\]
Therefore, 
\[
\int_0^t \ell(s)[(1\ast\ell)(s)]^{-\delta}g_{\gamma}(t-s)ds\lesssim t^{\gamma-1+\alpha-\frac{\alpha d}{\rho}\left(1-\frac{1}{p}\right)},\text{\ as \ }t\to\infty.
\]

We point out that in this example we have obtained the same rate of decay of \cite[Proposition 5.15]{Kem-Sil-Zach-2017}. However, we do not need to impose boundedness of $|f(t,\cdot)|_q$ near $0$.

\medskip

In the case $(ii)$ we do not know explicitly an expression for $\ell$. To find (\ref{Ex:decay:2}) we proceed as follows  
\[
\int_0^t \ell(s)[(1\ast\ell)(s)]^{-\delta}g_{\gamma}(t-s)ds=I_1(t)+I_2(t),
\]
where 
\[
I_1(t)=\displaystyle\int_{t/2}^t \ell(s)[(1\ast\ell)(s)]^{-\delta}g_{\gamma}(t-s)ds, \quad t>0,
\] 
and 
\[
I_2(t)=\displaystyle\int_0^{t/2} \ell(s)[(1\ast\ell)(s)]^{-\delta}g_{\gamma}(t-s)ds,\quad t>0.
\]

Since $s\mapsto\ell(s)[(1\ast\ell)(s)]^{-\delta}$ is non-increasing and non-singular in $[t/2,t]$, we have that 
\begin{align*}
I_1(t)&\lesssim\ell\Bigl(\frac{t}{2}\Bigr)\Bigl[(1\ast\ell)\Bigl(\frac{t}{2}\Bigr)\Bigr]^{-\delta}\int_{t/2}^t (t-s)^{\gamma-1}ds, \quad t>0\\
&\lesssim  \ell\Bigl(\frac{t}{2}\Bigr)\Bigl[(1\ast\ell)\Bigl(\frac{t}{2}\Bigr)\Bigr]^{-\delta}t^{\gamma}, \quad t>0
\end{align*}
As we have proved in Example \ref{Ex:2}, we have that $[(1\ast\ell)(t)]^{-1}\sim t^{-\alpha}$ as $t\to \infty$. Further, we have that $\ell(t)\sim t^{\alpha-1}$, as $t\to\infty$. Therefore, we have 
\[
I_1(t)\lesssim t^{\alpha-1-\alpha\delta +\gamma}.
\]
In what concerns to $I_2(t)$, we have that $s\mapsto g_{\gamma}(t-s)$ is increasing and non-singular in $[0,t/2]$. Thus,
\begin{align*}
I_2(t)&\lesssim t^{\gamma-1}\int_{0}^{t/2} \ell(s)[(1\ast\ell)(s)]^{-\delta}ds\\
&\lesssim t^{\gamma-1}\int_0^{c_0} \ell(s)[(1\ast\ell)(s)]^{-\delta}ds+t^{\gamma-1}\int_{c_0}^{\frac{t}{2}} \ell(s)[(1\ast\ell)(s)]^{-\delta}ds,\end{align*}
where $c_0$ satisfies that $\ell(t)\lesssim t^{\alpha-1}$ for $t>c_0$ and $[(1\ast \ell)(t)]^{-\delta}\lesssim t^{-\alpha\delta}$ for $t>c_0$.

It has been established in \cite[Theorem 3, Section 5, Chapter XIII]{Fell-1971} that  Karamata-Feller's Theorem remains valid when the roles of the origin and infinity are interchanged, that is, for $\lambda\to \infty$ and $t\to 0$. Therefore, since $0<\alpha<\beta<1$, and the Laplace transform of the kernel $\ell$ is given by 
\[
\widehat{\ell}(\lambda)=\dfrac{1}{\lambda^\alpha+\lambda^\beta},\quad \lambda>0.
\]
We have that 
\[
\ell(t)\sim t^{\beta-1}, \text{\ and\ \ } [(1\ast\ell)(t)]^{-\delta}\sim t^{-\beta\delta}, \text{\ as \ }t\to 0.
\]
It follows from a direct calculation that 
\[
\int_0^{c_0} s^{\beta-1}s^{-\beta\delta}ds, \text{\ \ is convergent}.
\]
By the limit comparison test, we have that  $\displaystyle\int_0^{c_0} \ell(s)[(1\ast\ell)(s)]^{-\delta}ds$, is convergent as well. The choice of $c_0$ implies that
\[
t^{\gamma-1}\int_{c_0}^{\frac{t}{2}} \ell(s)[(1\ast\ell)(s)]^{-\delta}ds\lesssim t^{\gamma-1} \int_{c_0}^{\frac{t}{2}} s^{\alpha-1}s^{-\alpha\delta}ds\lesssim t^{\gamma -1+ \alpha-\alpha\delta}.
\]

In consequence, we have 
\[
|u(t,\cdot)|_r\lesssim t^{\gamma-1+\alpha-\frac{\alpha d}{\rho}\left(1-\frac{1}{p}\right)}.
\]

\medskip

The proof of $(iii)$ is very similar to the proof of $(ii)$. To obtain \eqref{Ex:decay:3} we write
\[
\int_0^t \ell(s)[(1\ast\ell)(s)]^{-\delta}g_{\gamma}(t-s)ds=I_1(t)+I_2(t),
\]
where 
\[
I_1(t)=\displaystyle\int_0^{t/2} \ell(s)[(1\ast\ell)(s)]^{-\delta}g_{\gamma}(t-s)ds,
\] 
and 
\[
I_2(t)=\displaystyle\int_{t/2}^t \ell(s)[(1\ast\ell)(s)]^{-\delta}g_{\gamma}(t-s)ds.
\]

Since $s\mapsto\ell(s)[(1\ast\ell)(s)]^{-\delta}$ is non-increasing and non-singular in $[t/2,t]$, we have that 
\begin{align*}
I_2(t)&\lesssim \ell\Bigl(\frac{t}{2}\Bigr)\Bigl[(1\ast\ell)\Bigl(\frac{t}{2}\Bigr)\Bigr]^{-\delta}t^{\gamma}, \quad t>0
\end{align*}
As we have proved in Example \ref{Ex:3}, we have that $[(1\ast\ell)(t)]^{-1}\sim t^{\beta-\alpha-1}$ as $t\to \infty$, and $\ell(t)\sim t^{\alpha-\beta}$, as $t\to\infty$. Therefore, we have 
\[
I_2(t)\lesssim t^{(\alpha-\beta)(1-\delta)-\delta+\gamma}.
\]
Regarding to $I_1(t)$, we have that $s\mapsto g_{\gamma}(t-s)$ is increasing and non-singular in $[0,t/2]$. Thus,
\begin{align*}
I_1(t)&\lesssim t^{\gamma-1}\int_0^{c_0} \ell(s)[(1\ast\ell)(s)]^{-\delta}ds+t^{\gamma-1}\int_{c_0}^{\frac{t}{2}} \ell(s)[(1\ast\ell)(s)]^{-\delta}ds,\end{align*}
where $c_0>0$ satisfies that $\ell(t)\lesssim t^{\alpha-\beta}$ for $t>c_0$ and $[(1\ast \ell)(t)]^{-\delta}\lesssim t^{(\beta-\alpha-1)\delta}$ for $t>c_0$.

Since $0<\alpha<\beta<1$, and the Laplace transform of the kernel $\ell$ is given by 
\[
\widehat{\ell}(\lambda)=\dfrac{1}{\lambda^{2+\alpha-\beta}}(\lambda^\alpha+\omega)
\] 
Interchanging the roles of the origin and infinity, it follows from Karamata-Feller's Theorem  that 
\[
\ell(t)\sim t^{-\beta}, \text{\ and\ \ } [(1\ast\ell)(t)]^{-\delta}\sim t^{(\beta-1)\delta}, \text{\ as \ }t\to 0.
\]
Therefore, from a direct calculation we have  
\[
\int_0^{c_0} s^{-\beta}s^{(\beta-1)\delta}ds, \text{\ \ is a convergent improper integral}.
\]
By the limit comparison test of improper integrals, we have $\displaystyle\int_0^{c_0} \ell(s)[(1\ast\ell)(s)]^{-\delta}ds$, is convergent. 

The choice of $c_0$ implies that
\[
t^{\gamma-1}\int_{c_0}^{\frac{t}{2}} \ell(s)[(1\ast\ell)(s)]^{-\delta}ds\lesssim t^{\gamma-1}\int_{c_0}^{\frac{t}{2}} s^{\alpha-\beta}s^{(\beta-\alpha-1)\delta}ds\lesssim t^{\gamma -1+(\alpha-\beta+1)(1-\delta)}.
\]

In consequence, we have 
\[
|u(t,\cdot)|_r\lesssim t^{\gamma-1+\bigl(\beta-\alpha-1\bigr)\bigl(1-\frac{d}{\rho}\bigl(1-\frac{1}{p}\bigr)\bigr)}.
\]

\end{proof}

\begin{corollary}\label{Theo:Ex:decay:2}
Let $q\ge 1$ and  
\[
k(t)=\int_0^1g_{\alpha}(t)\alpha^n d\alpha,\quad t>0.
\]  
If $f(t,\cdot)\in L_1(\mathbb{R}^d)\cap L_q(\mathbb{R}^d)$ for all $t\ge 0$ and $|f(t,\cdot)|_q\lesssim g_{\gamma}(t)$, for some $\gamma\in(0,1)$, then 
\begin{equation}\label{log:est}
|u(t,\cdot)|_q\lesssim t^{\gamma-1}\log(t).
\end{equation}
\end{corollary}
\begin{proof}
By \cite[Proposition 3.1]{Koch-2008} there exists $\ell \in L_{1,loc}(\mathbb{R}^+)$ such that $(k\ast\ell)=1$. It follows from \eqref{Laplace:k1} 
\[
\widehat{\ell}(\lambda)=\dfrac{\log^{n+1}(\lambda)}{n!\left(\displaystyle\lambda-\sum_{m=0}^n \dfrac{\log^{m}(\lambda)}{m!}\right)}, \text{\ for\ }\lambda>0.
\]
Thus, we have that 
\[
\widehat{g_\gamma}(\lambda)\widehat{\ell}(\lambda)=\dfrac{1}{\lambda^{\gamma}}L\left(\dfrac{1}{\lambda}\right),\quad \lambda>0.
\]
where $L(t)=\left((-1)^{n+1}t\log^{n+1}(t)\right)\left(\displaystyle n!\left(1-\sum_{m=0}^n \dfrac{(-1)^m t\log^m(t)}{m!}\right)\right)^{-1}$. The function $L$ is slowly varying at $\infty$ and $L(t)\sim \log(t)$ as $t\to \infty$. Therefore, the decay rate \eqref{log:est} it follows from Theorem \ref{Karamata} and Theorem \ref{Theo:Lr:Est:f}.
\end{proof}

\begin{corollary}\label{Theo:Ex:decay:3} Let $1\le p\le \sigma_2(d,\rho)$ and $1<q,r<\infty$ be such that $1+\frac{1}{p}=\frac{1}{r}+\frac{1}{q}$. Assume further that $|f(t,\cdot)|_q\lesssim g_{\gamma}(t)$ for some $\gamma\in(0,1)$. The following assertions hold.
\begin{enumerate}[$(i)$]
\item If $k=g_{1-\alpha}$, then 
\[
|\nabla u(t,\cdot)|_r\lesssim t^{\gamma -1 +\alpha-\frac{\alpha}{\rho}-\frac{\alpha d}{\rho}\bigl(1-\frac{1}{p}\bigr)},\text{\ as \ }t\to\infty.
\] 
\item If $k=g_{1-\alpha}+g_{1-\beta}$, with $0<\alpha<\beta<1$, then 
\[
|\nabla u(t,\cdot)|_r\lesssim t^{\gamma -1 +\alpha-\frac{\alpha}{\rho}-\frac{\alpha d}{\rho}\bigl(1-\frac{1}{p}\bigr)},\text{\ as \ }t\to\infty.
\] 

\item If $k(t)=g_{\beta}(t)E_{\alpha,\beta}(-\omega t^{\alpha})$, for $t>0$,
with $0<\alpha<\beta<1$ and $\omega>0$, then 
\[
|\nabla u(t,\cdot)|_r\lesssim t^{\gamma-1+\bigl(\beta-\alpha-1\bigr)\bigl(1-\frac{1}{\rho}-\frac{d}{\rho}\bigl(1-\frac{1}{p}\bigr)\bigr)},\text{\ as \ }t\to\infty .
\]
\end{enumerate}
\end{corollary}

\begin{proof}
The proof follows from a slight modification of the proof of Theorem \ref{Theo:Ex:decay} and the application of Theorem \ref{Theo:Lr:Grad:Est:f}.
\end{proof}

We conclude pointing out that the method developed in this work can be applied for more situations. For instance, by switching the kernels $k$ and $\ell$ we obtain more interesting examples of pair $(k,\ell)\in(\mathcal{PC})$.

\medskip

We expect to generalize our methods to consider another evolutionary equations with nonlocal time diffusion. More specifically, we want to find optimal decay rates for solution to nonlocal differential equations with the the fractional $p$-Laplacian operator or the sum of different space-fractional operators. As well as, we want to consider the porous medium equation, the case of the doubly nonlinear equation, the case of the mean curvature equation. Recently, this has been successfully made for fractional in time evolutionary equations in \cite{Dip-Val-Ves-2017}.

\bibliographystyle{aims}
\bibliography{Ref}

\end{document}